\documentclass{amsart}

\usepackage[linktoc=page]{hyperref}

\usepackage[utf8]{inputenc}
\usepackage{cite}
\usepackage{amssymb}
\usepackage{amsthm}
\usepackage{amsfonts}
\usepackage{amsmath}
\usepackage{mathrsfs}
\usepackage[all]{xy}
\usepackage{graphicx}
\usepackage{mathrsfs}
\usepackage{extpfeil}
\usepackage{mathtools}
\usepackage{tikz-cd}
\usepackage{bm}
\usepackage{latexsym, amscd ,psfrag}
\usepackage{graphicx}
\usepackage{caption}
\usepackage{float}

\usepackage[mathscr]{euscript}
\usepackage{enumitem}
\usepackage{cleveref}

\makeatletter
\newsavebox{\@brx}
\newcommand{\llangle}[1][]{\savebox{\@brx}{\(\m@th{#1\langle}\)}%
  \mathopen{\copy\@brx\kern-0.5\wd\@brx\usebox{\@brx}}}
\newcommand{\rrangle}[1][]{\savebox{\@brx}{\(\m@th{#1\rangle}\)}%
  \mathclose{\copy\@brx\kern-0.5\wd\@brx\usebox{\@brx}}}
\makeatother
\def\dsum{\displaystyle\sum}
\def\dprod{\displaystyle\prod}
\def\dint{\displaystyle\int}

\newcommand{\ep}{\epsilon}

\newcommand{\Si}{\Sigma}
\newcommand{\Ga}{\Gamma}

\newcommand{\bC}{\mathbb{C}}
\newcommand{\CC}{\mathbb{C}}
\newcommand{\bE}{\mathbb{E}}
\newcommand{\bK}{\mathbb{K}}
\newcommand{\bL}{\mathbb{L}}

\newcommand{\bP}{\mathbb{P}}
\newcommand{\bQ}{\mathbb{Q}}
\newcommand{\bR}{\mathbb{R}}

\newcommand{\bT}{\mathbb{T}}
\newcommand{\TT}{\mathbb{T}}
\newcommand{\bZ}{\mathbb{Z}}
\newcommand{\ZZ}{\mathbb{Z}}

\newcommand{\cA}{\mathcal{A}}
\newcommand{\cB}{\mathcal{B}}

\newcommand{\cE}{\mathcal{E}}
\newcommand{\cF}{\mathcal{F}}

\newcommand{\cL}{\mathcal{L}}
\newcommand{\cI}{\mathcal{I}}
\newcommand{\cM}{\mathcal{M}}
\newcommand{\cO}{\mathcal{O}}

\newcommand{\cS}{\mathcal{S}}
\newcommand{\cQ}{\mathcal{Q}}

\newcommand{\cX}{\mathcal{X}}

\newcommand{\age}{\mathrm{age}}

\newcommand{{\inv} }{\mathrm{inv}}
\newcommand{\ev}{\mathrm{ev}}
\newcommand{\Aut}{\mathrm{Aut}}

\newcommand{\Res}{\mathrm{Res}}
\newcommand{\rank}{\mathrm{rank}}
\newcommand{\val}{ {\mathrm{val}} }
\newcommand{\vir}{{\mathrm{vir}}}
\newcommand{\CR}{  {\mathrm{CR}}  }
\newcommand{\Jac}{ {\mathrm{Jac}} }

\newcommand{\one}{\mathbf{1}}

\newcommand{\bu}{\mathbf{u}}

\newcommand{\bp}{\mathbf{p}}
\newcommand{\bq}{\mathbf{q}}
\newcommand{\bt}{\mathbf{t}}
\newcommand{\w}{\mathbf{w}}
\newcommand{\bw}{\mathbf{w}}

\newcommand{\su}{\mathsf{u}}

\newcommand{\txi}{ {\widetilde{\xi}} }

\newcommand{\tX}{{\widetilde{X}}}

\newcommand{\BG}{\mathcal{B} G}

\newcommand{\nov}{\Lambda_{\mathrm{nov}}}

\newcommand\fh{\mathfrak{h}}

\newcommand{\OGW}{\cF(m,r),S^{1},T}
\newcommand{\CGW}{\cF(m,r),T}

\newcommand{\eff}{\text{eff}}
\newcommand{\ualpha}{\underline{\alpha}}
\newcommand{\ubeta}{\underline{\beta}}

\newtheorem{lma}{Lemma}[section]
\newtheorem{coro}[lma]{Corollary}
\newtheorem{defn}[lma]{Definition}
\newtheorem{prop}[lma]{Proposition}

\newtheorem{theorem}[lma]{Theorem}
\newtheorem{remark}[lma]{Remark}

\theoremstyle{definition}

\newtheorem{convention}[lma]{Convention}

\begin{document}

\title{All genus open mirror symmetry for footballs}

\author{Zhuoming Lan}
\address{Zhuoming Lan, Beijing International Center for Mathematical Research, Peking University, 5 Yiheyuan Road, Beijing 100871, China}
\email{lanzm26@pku.edu.cn}
\author{Jinghao Yu}
\address{Jinghao Yu, School of Mathematical Sciences, Peking University, Haidian District, Beijing 100871, China}
\email{yujinghao@math.pku.edu.cn}

\author{Zhengyu Zong}
\address{Zhengyu Zong, Department of Mathematical Sciences,
	Tsinghua University, Haidian District, Beijing 100084, China}
\email{zyzong@mail.tsinghua.edu.cn}

\maketitle

\begin{abstract}
We prove an all genus full descendant open mirror symmetry for footballs. The B-model is given by the Chekhov-Eynard-Orantin topological recursion on the mirror curve.
\end{abstract}

\tableofcontents

\section{Introduction}

\subsection{Historical background and motivation}

\subsubsection{Mirror symmetry for orbifolds in the closed string sector}
Mirror symmetry is a duality from string theory originally discovered by physicists. It asserts that Type IIA and Type IIB string theories on distinct Calabi--Yau threefolds yield identical physical theories. This physical correspondence captured the attention of mathematicians in the early 1990s, particularly following the seminal work of Candelas, de la Ossa, Green, and Parkes \cite{CdGP}. By computing period integrals on the mirror manifold, they successfully formulated a striking conjecture for the enumeration of rational curves of all degrees on the quintic threefold.

By the late 1990s, a rigorous mathematical framework for the A-model topological closed string was established through the development of Gromov-Witten (GW) theory. Within this framework, generating functions of genus $g$ GW invariants serve to define the genus $g$ topological free energy. The mathematical aspect of genus zero closed mirror symmetry has been well-studied in many cases. Givental \cite{G96} and Lian-Liu-Yau \cite{LLY97} independently proved the genus zero mirror formula for the quintic Calabi-Yau 3-fold $Q$. Subsequent works by these authors generalized the correspondence to Calabi--Yau complete intersections within projective toric varieties \cite{G98, LLY99, LLY3}. 

The mathematical theory of A-model topological strings on orbifolds is given by orbifold GW theory \cite{AGV02,AGV08,CR02,CR04,Ts10}. The genus-zero mirror theorem for toric Deligne-Mumford stacks is proved in \cite{CCIT15, CCK}.

Extending mirror symmetry to higher genera presents substantially deeper challenges. In the case of toric Calabi-Yau 3-folds/3-orbifolds, Bouchard-Klemm-Mari\~{n}o-Pasquetti  (BKMP) \cite{BKMP09, BKMP10} proposed a new formalism of the topological B-model in terms of the Chekhov-Eynard-Orantin invariants of the mirror curve based on the work of Eynard-Orantin \cite{EO07} and Mari\~{n}o \cite{Ma}. BKMP conjectured a precise correspondence, known as the BKMP Remodeling Conjecture, between the local expansion of the Chekhov-Eynard-Orantin (CEO) invariants at the puncture of the mirror curve and the generating function of open Gromov-Witten invariants of toric Calabi-Yau 3-folds/3-orbifolds. The mathematical study of the Remodeling Conjecture can be found, for example, in \cite{BCMS,Ch09,EO15,FL,FLT,FLZ20a,FLZ20b,Zh09,Zh09b,Zh10,Zhu}. Extending this formalism to \emph{compact} target spaces introduces great complications. Bershadsky-Cecotti-Ooguri-Vafa (BCOV) conjectured the genus-one and genus-two mirror formulae for the quintic 3-fold \cite{BCOV}. Combining the techniques of BCOV, results of Yamaguchi-Yau \cite{YY04}, and boundary conditions, Huang-Klemm-Quackenbush \cite{HKQ} proposed a mirror conjecture on $F_g^Q$ up to $g=51$. The BCOV genus-one  mirror formula was first proved by A. Zinger in \cite{Zi09} using genus-one reduced Gromov-Witten theory, and later reproved in \cite{KimL, CFKim} via quasimap theory and in \cite{CGLZ} via MSP theory. The BCOV genus-two mirror formula was proved by Guo-Janda-Ruan \cite{GJR} and Chang-Guo-Li \cite{CGL}. The Yamaguchi--Yau's finite generation and the holomorphic anomaly equation for quintic 3-folds were proved in \cite{CGL} and \cite{GJR2}.

In contrast to Calabi--Yau threefolds, all-genus mirror symmetry for the projective line is more tractable. By connecting the Chekhov--Eynard--Orantin (CEO) topological recursion with the higher-genus closed GW invariants of $\bP^1$, Dunin-Barkowski, Orantin, Shadrin, and Spitz \cite{DOSS} successfully proved the Norbury--Scott conjecture \cite{NS}. This correspondence was later generalized to the equivariant setting in \cite{FLZ17}. The main result of \cite{FLZ17} relates the higher genus equivariant descendant Gromov-Witten potentials of $\bP^1$ to the oscillatory integrals of Chekhov-Eynard-Orantin invariants of the mirror curve. This structural identification can be viewed as an all-genus equivariant closed mirror theorem for $\bP^1$, which recovers the Norbury-Scott conjecture by taking the non-equivariant limit. The result of \cite{FLZ17} is generalized to the case of weighted projective line in \cite{Tang} and to the case of footballs in \cite{Lan25}.

\subsubsection{All genus open mirror symmetry for the projective line}
While the closed string sector of orbifolds has seen tremendous progress, the open string sector is much more complicated. Open Gromov-Witten theory studies the enumeration of maps from bordered Riemann surfaces to a target space with boundary conditions mapped to Lagrangian submanifolds. The all genus $S^1$-equivariant open Gromov-Witten theory of $(\bP^1,\bR\bP^1)$ is studied via coherent boundary conditions by Buryak-Netser Zernik-Pandharipande-Tessler in \cite{BNPT22}. The all genus open mirror symmetry for $(\bP^1,\bR\bP^1)$ is proved in \cite{YZ}, extending the result in \cite{FLZ17} to the open string sector. On the B-model side, instead of taking the oscillatory integrals, the expansion of the Chekhov-Eynard-Orantin invariants under certain local coordinate at the puncture of the mirror curve is considered.

For an orbifold target like the football $\mathcal{F}(m,r)$ (see Section \ref{sec:fan} for the definition), the geometry of the moduli space of open stable maps is enriched by the presence of stacky points. Computing these open orbifold invariants is a highly non-trivial task. Consequently, there is a strong motivation to find a recursive algorithmic structure to compute these all-genus open orbifold invariants effectively.

\subsubsection{All genus open mirror symmetry for the footballs}
The primary motivation of this paper is to bridge the gap between open orbifold Gromov-Witten theory and the topological B-model by generalizing the all-genus open mirror symmetry of $\mathbb{P}^1$ to the football $\mathcal{F}(m,r)$. In the previous work \cite{YZ}, the second and the third author showed that the local expansion of CEO invariants on the mirror curve exactly recovers the open GW invariants of the smooth projective line. It is a natural question whether the BKMP remodeling philosophy survives the introduction of stacky indices $m$ and $r$ on a compact target space.

In this paper, we answer this question in the affirmative. We construct the higher-genus equivariant open Gromov-Witten theory for the pair $(\mathcal{F}(m,r), S^1)$ via integration over the fixed locus, establishing the foundational virtual localization formulas and open/descendant correspondences. On the B-model side, we utilize the CEO topological recursion on the mirror curve of the football. Our main results demonstrate that the B-model open potentials precisely match the generating functions of the open GW invariants of $(\mathcal{F}(m,r), S^1)$ under the mirror map. 

Furthermore, we push this symmetry to its ultimate conclusion by establishing the \textit{full descendant open mirror symmetry}. By carefully applying oscillatory integrals over the SYZ dual cycles of equivariant line bundles, we prove that the topological recursion completely governs the open orbifold GW invariants with arbitrary descendant insertions. This result not only provides a powerful, effective algorithm for computing the higher-genus open invariants of $\mathcal{F}(m,r)$, but it also deepens our understanding of how mirror symmetry interacts with the descendant insertions and twisted sectors.

In the end, we would like to give a remark on our definition of the open GW invariants of $(\mathcal{F}(m,r), S^1)$. In \cite{BNPT22}, a graph sum formula for the $S^1$-equivariant open Gromov-Witten theory of $(\bP^1,\bR\bP^1)$ is introduced and this graph sum formula is conjectured to have a geometric definition via coherent boundary conditions. The vertex factor in this graph sum formula is given by the open Gromov-Witten invariants defined via integration over the fixed locus. In our paper, the open GW invariants of $(\mathcal{F}(m,r), S^1)$ are defined via integration over the fixed locus. We expect that there should be a geometric definition of the open GW invariants of $(\mathcal{F}(m,r), S^1)$ via coherent boundary conditions. Then the open GW invariants studied in our paper would play the role of the building block and the mirror theorem in our paper would imply the all genus open symmetry for the geometric open GW invariants. In the case of $\bP^1$, this is proved in \cite[Theorem 6.9]{YZ}.

\subsection{Main Results}
In this paper, we denote the 1-dimensional toric orbifold defined by toric fan with cones $-m$ and $r$ by $\cF(m,r)$.
Let $L\simeq S^{1}$ be the equator Lagrangian of $\cF(m,r)$. Let $T$ be the homotopy subtorus of $(\bC^{*})^{m+r+1}$ that fixes the $L$ (see Section \ref{sec:fan} for details).
The $T$-equivariant Chen-Ruan cohomology of $\cF(m,r)$ is:
$$H^{*}_{\CR,T}(\cF(m,r),\bC)=\bC[y_{1},y_{2},\bw]/(ry_{1}^{r}-my_{2}^{m}-\bp,y_{1}y_{2})$$
where $\bw=(\bw_{-m},\dots,\bw_{r})$ are equivariant parameters of $T$.

In Section \ref{sec:gen.fun.}, we will define the generating function $F_{g,h}(\bt,Q;X_{1},\dots, X_{h})$ of
genus $g$, $h$ boundary circles open Gromov-Witten invariants of $(\cF(m,r)
, S^{1}
)$. Here $\bt=tH+\dsum_{i=-m+1}^{r-1}t^{i}\one_{i}$ (see Section \ref{sec: cohomology})
, $Q$ is the Novikov variable encoding the degree of the stable maps to $\cF(m,r)$
, and $X_1,\dots, X_h$ are variables encoding the winding numbers (viewed as the
open string coordinates).

Our main theorem relates equivariant open Gromov-Witten invariants of $(\cF(m,r),S^{1})$
to the Eynard-Orantin invariants \cite{EO07} of
the mirror curve 
$$\{(x,Y)\in\bC\times\bC^{*}|x=Y^r + \sum_{l=-m}^{r-1} \tilde{q}_lY^l - \bw_r\log(Y^r) - \sum_{l=-m}^{r-1}\bw_l\log(\tilde{q}_lY^l)\}$$
where $\tilde{q}_{i}$
are complex parameters (see Section \ref{sec:W_T} and Section \ref{sec:curve}).

We study the Chekhov-Eynard-Orantin invariants $\omega_{g,n}$ of the
mirror curve. Then we use the full integral of $\omega_{g,n}$ to define the B-model potential $W_{g,n}(\bq,X_{1},\dots,X_{n})$ (see Section \ref{sec}).

 The statement of all genus open mirror symmetry is as follows.
\begin{theorem}[see also Theorem \ref{thm:main}]
Letting $X_{i}=e^{-\frac{W_{T}(Y_{i})}{\bp}}$, for $g \geq 0, n>0$, we have:
   \[
W_{g,n}(\bold{q},X_{1},\dots,X_{n})=F_{g,n} (\tau(\bold{q}),1;X_1,\dots,X_n)
\] 
where $\tau$ is the mirror map.
\end{theorem}
We use the integral of last $n_{2}$ directions of $\omega_{g,n_{1}+n_{2}}$ to define the B-model potential $W_{g,(n_{1},n_{2})}$, and consider its Laplace transformations on the SYZ dual of a line bundle $\cL$ (see Section \ref{sec:full descendant MS} and \cite{Fang20}), we produce the open mirror symmetry for full descendants.
\begin{theorem}[See also Theorem \ref{thm:full descendant MS}]
    For $g, n_1,n_2\geq 0, n_1+n_{2}>0$, and $\cL_{i}\in K_{T}(\cF(m,r))$ for $i = 1,\dots, n_1$, we have
    \begin{eqnarray*}        
        &\dint\dots\dint_{y_{i}\in\mathrm{SYZ}(\cL_{i})} \exp(\dsum_{i=1}^{n_{1}}\frac{W_{T}(y_{i})}{z_{i}}) W_{g,(n_{1},n_{2})}(y_{1},\dots,y_{n_{1}},X_{n_{1}+1},\dots,X_{n_{1}+n_{2}}) \\
    &=\dsum_{\vec{\mu}\in(\bZ_{\neq0})^{n_{2}}}\left<\!\left< \dfrac{\kappa(\mathcal{L}_1)}{z_1-\psi_1}, \cdots,\dfrac{\kappa(\mathcal{L}_{n_{1}})}{z_{n_{1}}-\psi_{n_{1}}}  \right>\!\right>_{g,\vec{\mu}}^{\OGW}\dprod_{i=1}^{n_{2}}X_{i+n_{1}}^{\mu_{i}}.
    \end{eqnarray*}
where $\kappa(\cL_{i})$ is the equivariant $K$-theoretic framing of $\cL_{i}$ (see Section \ref{sec:full descendant MS}).

For $n_{1}=0$ this theorem reduces to Theorem \ref{thm:main} and for $n_{2}=0$ the theorem reduces to closed full descendant mirror symmetry proved in \cite[Section 4.2]{Lan25}.
\end{theorem}
\subsection*{Acknowledgements}
The authors would like to thank Bohan Fang, Chiu-Chu Melissa Liu, and Song Yu for useful discussions. The third author is partially supported by the Natural Science Foundation of Beijing, China grant No. 1252008 and NSFC (grant No. 12571067).

\section{Equivariant closed Gromov-Witten theory of $\cF(m,r)$}\label{sec:closedGW}
In this section, we recap the geometry and closed Gromov-Witten invariants of $\cF(m,r)$, which was introduced in detail in \cite{Lan25}. 
\subsection{Geometry of $\mathcal{F}(m,r)$ as a toric orbifold}\label{sec:fan}

\color{black}
A football $\mathcal{F}(m,r)$ is a toric orbifold defined by the following stacky fan \cite{BCS05}: 
\begin{figure}[h]
\begin{center}
\setlength{\unitlength}{2mm}
\begin{picture}(20,5)
\put(12,4){\vector(-1,0){10}}
\put(12,4){\vector(1,0){6}}
\put(12,3){${}^|$}
\put(1,1){$-m$}
\put(17,1){$r$}
\end{picture}
  \label{fig0}
\end{center}
\end{figure}
\vspace{-5mm}

\noindent i.e., $\mathcal{F}(m,r)=(\mathbb{C}^2\setminus\{(0,0)\})/G_{m,r}$, where $G_{m,r}$ is defined by $$G_{m,r}=\{(t_1,t_2)|(t_1,t_2)\in(\mathbb{C}^{*})^2,t_1^{r} t_2^{-m}=1\}.$$ From this construction we have a $T=(\mathbb{C}^\ast)^2$ action on $\mathcal{F}(m,r)$ by $(t_1,t_2)\cdot[z_1,z_2]=[t_1z_1,t_2z_2]$.

The football $\cF(m,r)$ can also be described
in terms of an extended stacky fan \cite{Jiang08}. By this construction, $\cF(m,r)$ can be written as a quotient of connected torus group:
\begin{equation*}
    \cF(m,r)=[U_{\cA}/K]
\end{equation*}
with $$U_{\cA}=\{(t_{-m},\dots,t_{r})|(t_{-m},t_{r})\in\bC^{2}-\{0\},(t_{-m+1},\dots,t_{r-1})\in(\bC^{*})^{m+r-1}\},$$
$$K=\{(t_{-m},\dots,t_{r})\in (\bC^*)^{m+r+1}|\prod_{i=-m}^{r}t_{i}^{i}=1\}.$$
More details are explained in \cite{Lan25}.
\subsection{Equivariant cohomology 
of $\mathcal{F}(m,r)$}\label{sec: cohomology}

Let $\bT=(\bC^{*})^{m+r+1}$ be the dense torus of $U_{A}$, $\bT\to\cQ:=\bT/K$ be the quotient map. Let $\bullet$ denote a point. The surjective group homomorphism $\TT\to \cQ$ induces an injective ring homomorphism
\begin{eqnarray*}
H^*_{\cQ}(\bullet;\CC) = \CC[p] &\longrightarrow&  
H^*_{\TT}(\bullet;\CC) =\mathbb{C}[\bold{w}_{-m},\cdots,\w_{-1},\w_{0},\w_1,\cdots,\w_{r}] =: \CC[\w]\\
p &\longmapsto & \mathbf{p} 
= \sum_{\ell=1}^{r} \ell \bold{w}_\ell - \sum_{\ell=1}^m \ell \bold{w}_{-\ell} =
\bold{p}_+ -\bold{p}_-
\end{eqnarray*}
where $\mathbf{p}_{+} =\sum_{\ell=1}^{r} \ell \bold{w}_\ell$ and 
$\bold{p}_{-}=\sum_{\ell=1}^m \ell \bold{w}_{-\ell}$.

The $\cQ$-equivariant and $\TT$-equivariant cohomology rings
of $\cF(m,r)$ are given by (see e.g. \cite[Section 4.3]{CIJ}): 
\begin{align*}
    H^{*}_{\cQ}(\cF(m,r),\CC)&=\CC[u_1,u_2, p]/\langle u_1 u_2, p-(ru_1-mu_2) \rangle ,\\
    H^{*}_{\TT}(\cF(m,r),\CC)&=
    H^*_{\cQ}(\cF(m,r);\CC)
    \otimes_{\CC[p]} \CC[\w] =
    \CC[u_1, u_2, \mathbf{w}]/\langle u_1 u_2, \bold{p}-(ru_1-mu_2) \rangle.
\end{align*}

The $\cQ$-equivariant and $\TT$-equivariant
Chen-Ruan orbifold
cohomology rings of 
$\cF(m,r)$ are given by (see e.g. \cite[Section 8.8]{Liu13}):
\begin{align*}
    H^{*}_{\CR,\cQ}(\cF(m,r);\CC)&=\CC[y_1,y_2, p]/\langle y_1 y_2, p-(ry_1^{r}-my_2^{m}) \rangle, \\
    H^{*}_{\CR,\TT}(\cF(m,r);\CC)&=
    H^*_{\CR,\cQ}(\cF(m,r);\CC)\otimes_{\CC[p]}\CC[\w]=
    \CC[y_1, y_2, \mathbf{w}]/\langle y_1 y_2, \bold{p}-(ry_1^{r}-my_2^{m}) \rangle.
\end{align*}

Let 
$$   
    L := \{[z_{1},z_{2}]\in\cF(m,r): |z_{1}|=|z_{2}|=1\}
$$
be the Lagrangian submanifold of $\cF(m,r)$. 
Consider the subset of $T\subset\bT$ that fixes $L$. According to the construction above, we can also explicitly define:
\begin{equation*}
    T=\{(t_{-m},\dots,t_{r})|\dprod_{i=-m}^{r}|t_{i}^{i}|=1\}.
\end{equation*}

The $T$-equivariant Chen-Ruan orbifold cohomology ring of $\cF(m,r)$ is isomorphic to the $\bT$-equivariant counterpart:
\[
    H^*_{\CR,T}(\cF(m,r);\bC) = \bC[y_{1},y_{2},\w]/\langle (y_{1}y_{2},\bp-(ry_{1}^{r}-my_{2}^{m}))\rangle.
\]
Let $p_r$ be the $r$-orbifold point $[0,1]$ and $p_{-m}$ be the $m$-orbifold point $[1,0]$. Both of them are $T$-fixed points. Let $G_{-m}, G_r$ be the monodromy groups of $p_{-m}, p_r$  ($G_{-m}=\bZ_m, G_r=\bZ_r$) respectively. Let $H=ry_{1}^{r}-\bp_{+}=my_{2}^{m}-\bp_{-}$.

Let $\cI\cF(m,r)$ be the inertia stack of $\cF(m,r)$. We have
\[
 \cI\cF(m,r) = \bigsqcup_{v\in\bZ\atop{-m < v < r}} \cF_v
\]
where 
\begin{equation*}
    \cF_v = \left\{\begin{aligned}
        &\cB\bZ_r, \quad &&1 \leq v\leq r-1,
        \\
        &\cF(m,r), \quad &&v=0,
        \\
        & \cB\bZ_m, \quad && -m+1\leq v \leq -1.
    \end{aligned}\right.
\end{equation*}
For convenience, we define $\text{Box}(\Si) = \{v\in\bZ:-m<v<r\}$, where $\Si$ is the stacky fan of $\cF(m,r)$.
As a graded vector space, the Chen-Ruan orbifold cohomology rings of $\cF(m,r)$ is identified by
\[
    \begin{split}
        H^*_{\CR,T}(\cF(m,r);\bC) &= H^*(\cF(m,r))\oplus\bigoplus_{i=1}^{r-1} H^*(\cF_i)[\frac{2i}{r}]\oplus \bigoplus_{j=1}^{m-1}H^*(\cF_{-j})[\frac{2j}{m}]
        \\
        &= \bC\one_0\oplus \bC H \oplus \bigoplus_{i=1}^{r-1}\bC\one_{i}\oplus\bigoplus_{j=1}^{m-1}\bC\one_{-j},
    \end{split}
\]
where $\one_i = y_1^i$ for $1\leq i\leq r-1$, and $\one_{-j}= y_2^j$ for $1\leq j\leq m-1$. 
Let $(,)_{\cF(m,r),T}$ be the $T$-equivariant Chen-Ruan orbifold Poincar\'{e} pairing of $H^*_{\CR,T}(\cF(m,r);\bC)$.

\subsection{Equivariant closed Gromov-Witten invariants of $\cF(m,r)$}
Let $E(\cF(m,r))$ denote the effective curve classes in $H_2(\cF(m,r);\bZ)$.
Given nonnegative integers $g, n$ and an effective curve class $\beta\in E(\cF(m,r))$,
let $\overline{\cM}_{g,n}(\cF(m,r),\beta)$ be the moduli stack of genus $g$, $n$-pointed, degree $\beta$ stable maps to $\cF(m,r)$.
Let $\ev_i: \overline{\cM}_{g,n}(\cF(m,r),\beta)\rightarrow \cI\cF(m,r)$ be the evaluation map at the $i$-th marked point. The $T$-action on $\cF(m,r)$ induces a $T$-action 
on $\overline{\cM}_{g,n}(\cF(m,r),\beta)$ and the evaluation map $\ev_i$ is $T$-equivariant.

For $i=1,\dots,n$, let $\bL_i$ be the $i$-th tautological line bundle over $\overline{\cM}_{g,n}(\cF(m,r),\beta)$ formed by the cotangent line at the $i$-th marked point.
Define the $i$-th descendant class $\psi_i$ as 
$$
    \psi_i := c_1(\bL_i)\in H^2(\overline{\cM}_{g,n}(\cF(m,r),\beta);\bQ).
$$
We choose an $T$-equivariant lift $\psi_i^{T}\in H^2_{T}(\overline{\cM}_{g,n}(\cF(m,r),\beta);\bQ)$ of $\psi_i$.

Given $\gamma_1,\dots,\gamma_n\in H^*_{\CR,T}(\cF(m,r);\bC)$ and nonnegative integers $a_1,\dots, a_n$, we define genus-$g$, degree-$\beta$, $T$-equivariant descendant 
Gromov-Witten invariants of $\cF(m,r)$:
\begin{align*}
    \langle \tau_{a_1}(\gamma_1),\dots,\tau_{a_n}(\gamma_n)\rangle_{g,n,\beta}^{\cF(m,r),T} &= 
    \langle \gamma_1\psi^{a_1},\dots,\gamma_n\psi^{a_n}\rangle_{g,n,\beta}^{\cF(m,r),T}
    \\
    & := \int_{[\overline{\cM}_{g,n}(\cF(m,r),\beta)]^{\vir,T}}
    \prod_{i=1}^n \ev_i^*(\gamma_i)(\psi_i^{T})^{a_i} \in \bC[{\bf w}],
\end{align*}
where $[\overline{\cM}_{g,n}(\cF(m,r),\beta)]^{\vir,T}$ is the weighted virtual fundamental class of $\overline{\cM}_{g,n}(\cF(m,r),\beta)$ \cite{AGV08,AGV02}.
Define the Novikov ring 
\[
    \nov := \widehat{\bC[E(\cF(m,r))]} = \Bigg\{\sum_{\beta\in E(\cF(m,r))}c_\beta Q^\beta: c_\beta\in\bC\Bigg\},
\]
where $Q$ is the Novikov variable. 
Let $\bt = tH+\dsum_{i=-m+1}^{r-1}t^{i}\bold{1}_{i}\in H^*_{\CR,T}(\cF(m,r);\bC)$, we define the following double correlator:
\[
    \llangle\gamma_1\psi^{a_1},\dots,\gamma_n\psi^{a_n}\rrangle_{g,n}^{\cF(m,r),T} := 
    \sum_{\beta\in E(\cF(m,r))}\sum_{l=0}^\infty \frac{Q^\beta}{l!}\langle \tau_{a_1}(\gamma_1),\dots,\tau_{a_n}(\gamma_n)
    ,\bt^l\rangle_{g,n+l,\beta}^{\cF(m,r),T}.
\]

For $j=1,\dots,n$, introduce formal variables
\[
    {\bf{u}}_j = {\bf{u}}_j(z) = \sum_{a\geq 0}(u_j)_az^a
\]
where $(u_j)_a\in H^*_{\CR,T}(\cF(m,r))$. Define
\begin{align*}
    \llangle {\bf u}_1,\dots, {\bf u}_n \rrangle_{g,n}^{\cF(m,r),T} &=  \llangle {\bf u}_1(\psi),\dots, {\bf u}_n(\psi)\rrangle_{g,n}^{\cF(m,r),T}
    \\
    &= \sum_{a_1,\dots,a_n\geq 0}\llangle (u_1)_{a_1}\psi^{a_1},\dots, (u_n)_{a_n}\psi^{a_n}\rrangle_{g,n}^{\cF(m,r),T}.
\end{align*}

Let $z_1,\dots,z_n$ be formal variables and $\gamma_1,\dots,\gamma_n\in H^*_{\CR,T}(\cF(m,r))$. Define
\[
    \llangle \frac{\gamma_1}{z_1-\psi},\dots,\frac{\gamma_n}{z_n-\psi}\rrangle_{g,n}^{\cF(m,r),T} = 
    \sum_{a_1,\dots,a_n\in\bZ_{\geq 0}}\llangle\gamma_1\psi^{a_1},\dots,\gamma_n\psi^{a_n}\rrangle_{g,n}^{\cF(m,r),T}\prod_{i=1}^{n}z_i^{-a_i-1}.
\]
We use the conventions that
\[
    \langle\frac{\gamma}{z-\psi}\rangle_{0,1,0}^{\cF(m,r),T} := z(\one, \gamma)_{\cF(m,r),T},
\]
\[
    \langle\frac{\gamma_1}{z-\psi},\gamma_2\rangle_{0,2,0}^{\cF(m,r),T} := (\gamma_1,\gamma_2)_{\cF(m,r),T},
\]
\[
    \langle\frac{\gamma_1}{z_1-\psi_1},\frac{\gamma_2}{z_2-\psi_2}\rangle_{0,2,0}^{\cF(m,r),T} := \frac{1}{z_1+z_2}(\gamma_1,\gamma_2)_{\cF(m,r),T}.
\]

\subsection{Equivariant quantum cohomology and Frobenius structure}
This section follows the notations and results in \cite{Lan25}.

The $T$-equivariant quantum cohomology ring $QH_T^*(\cF(m,r))$ of $\cF(m,r)$ is defined by its genus-zero primary Gromov-Witten invariants. As a $\bC[\bp]$-module, 
$QH_T^*(\cF(m,r)) = H^*_{\CR,T}(\cF(m,r))$. The ring structure is given by the quantum product $\star$:
\[
    (\gamma_1\star\gamma_2,\gamma_3)_{\cF(m,r),T} = \llangle \gamma_1,\gamma_2,\gamma_3\rrangle_{0,3}^{\cF(m,r),T}.
\]

Let $q_{i}$ and $q_{*}$ be as defined in \cite[Section 2.7]{Lan25}, the $T$-equivariant quantum cohomology ring of $\cF(m,r)$ is:
\[
    QH^*_T(\cF(m,r);\bC) = \bC[X,X^{-1},\bp,q]/\langle \dsum_{i=-m}^{r}i\tilde{q}_{i}X^{i}-\bp\rangle,
\]
where $\tilde{q}_r=1$, $\tilde{q}_{i}=q_{i}$ for $0\leq i\leq r-1$, $\tilde{q}_{i}=q_{*}^{-i}q_{i}$ for $-m+1\leq i<0$ and $\tilde{q}_{-m}=q_*^m$.

Let $\{z_{\alpha}\}$ be the roots of $\dsum_{i=-m}^{r}i\tilde{q}_{i}X^{i}-\bp$ for $\alpha=0,1\dots, m+r-1$, and let
\[
    \phi_{\alpha}(\bq) = \dprod_{\beta\neq\alpha}\frac{X-z_{\beta}}{z_{\alpha}-z_{\beta}}\in QH_{T}^{*}(\cF(m,r),\bC),
\]
Then 
\[
    \phi_\alpha(\bq)\star \phi_\beta(\bq) = \delta_{\alpha\beta}\phi_\alpha(\bq), \quad
    (\phi_\alpha(\bq), \phi_\beta(\bq))_{\cF(m,r),T} = \frac{\delta_{\alpha\beta}}{\Delta^\alpha(\bq)},
\]
where
\[
    \Delta^\alpha(\bq) = \frac{r\prod_{\beta\neq\alpha}(z_\alpha-z_\beta)}{z_\alpha^{m-1}}.
\]

The normalized canonical basis is defined by
\[
    \{ \hat{\phi}_\alpha(\bq) := \sqrt{\Delta^\alpha(\bq)}\phi_\alpha(\bq) : \alpha = 0,1\dots,m+r-1\}.
\]
They satisfy
\[
    \hat{\phi}_\alpha(\bq) \star \hat{\phi}_\beta(\bq) = \delta_{\alpha\beta}\sqrt{\Delta^\alpha(\bq)}\hat{\phi}_\alpha(\bq), \quad (\hat{\phi}_\alpha(\bq),\hat{\phi}_\beta(\bq))_{\cF(m,r),T} = \delta_{\alpha\beta}.
\]

\subsection{$\Psi$-matrix}\label{sec:A-canonical}
This section follows the notation and results in \cite{Lan25}. Consider the flat coordinate $\bt=\dsum_{i=0}^{m+r-1}t^{i}(\bold{q})X^{i}$, where $X=\one_{1}$ as an element in quantum cohomology ring, and canonical coordinate $\bu=\dsum_{\alpha=0}^{m+r-1}u^{\alpha}\phi_{\alpha}(\bq)$. In this subsection we calculate $\Psi$-matrix under the flat coordinates and canonical coordinates chosen in the previous subsection. 

By direct calculation, we have: 
\begin{eqnarray*}
\big(\frac{\partial}{\partial\bold{t}}\big)^{T}=&V(z_{0},z_{1},\dots,z_{m+r-1})\big(\frac{\partial}{\partial\bold{u}}\big)^{T}\\
\big({d\bold{t}}\big)^{T}=&V^{-1}(z_{0},z_{1},\dots,z_{m+r-1})^{T}\big({d\bold{u}}\big)^{T},
\end{eqnarray*}
where $V(z_{0},z_{1},\dots,z_{m+r-1})$ is the Vandermonde matrix with parameters $(z_{0},\dots,z_{m+r-1})$.
The above equations determine the canonical coordinates
$\bold{u}$ up to a constant in $\CC[\w]$.

We write $(\frac{\partial}{\partial\bold{t}})^{T}=V(z_{0},z_{1},\dots,z_{m+r-1})(\frac{\partial}{\partial\bold{u}})^{T}$. Then we have $$({d\bold{t}})^{T}=V^{-1}(z_{0},z_{1},\dots,z_{m+r-1})^{T}({d\bold{u}})^{T}.$$

For $\alpha\in \{0,1,\dots,m+r-1\}$ and $i\in \{0,1,\dots,m+r-1\}$, define $\Psi_i^{\,\ \alpha}$ by
$$
\frac{du^\alpha}{\sqrt{\Delta^\alpha(\bq)}} =\sum_{i=0}^{m+r-1} dt^i \Psi_i^{\,\ \alpha},
$$
and define the $\Psi$-matrix to be
$$
\Psi:= \left(\Psi_i^{\,\ \alpha} \right)_{i,\alpha}
$$
with row index $i$ and column index  $\alpha$. We can directly find out: 
$$\Psi=V(z_{0},z_{1},\dots,z_{m+r-1})\mathbf{diag}((\Delta^{0}(\bq))^{-\frac{1}{2}},
\dots,(\Delta^{m+r-1}(\bq))^{-\frac{1}{2}}).$$
This also explicitly gives  $\Psi_i^{\,\ \alpha}=z_{\alpha}^{i}(\Delta^{\alpha}(\bq))^{-\frac{1}{2}}$.

Let $\Psi^{-1}$ be the inverse matrix of $\Psi$, so we have:
$$
\Psi^{-1}=\mathbf{diag}((\Delta^{0}(\bq))^{\frac{1}{2}},
\dots,(\Delta^{m+r-1}(\bq))^{\frac{1}{2}})V^{-1}(z_{0},z_{1},\dots,z_{m+r-1}).
$$
We can also explicitly write:
$$(\Psi^{-1})^{\,\,i}_{\alpha}:=(\Psi^{-1})_{\alpha,i}=(-1)^{m+r-1-i}(\frac{r}{z_{\alpha}^{m-1}})(\Delta^{\alpha}(\bq))^{-\frac{1}{2}}\dsum_{J\subset S\setminus\{\alpha\},|J|=m+r-i-1}z_{J}$$
where $S=\{0,\dots,m+r-1\}$, $z_{J}=\dprod_{\alpha\in J}z_{\alpha}$.
\subsection{The $\cS$-operator}\label{sec:A-S}
In this section we set the notations of $\cS$-operator and $S$-matrice under different basis. 

For any cohomology classes $a,b\in H_{\CR,T}^*(\cF(m,r);\CC)$, we define: 
\begin{align*}
(a,\cS(b)(z))&=\llangle a,\frac{b}{z-\psi}\rrangle^{\cF(m,r),T}_{0,2},
\\
 V_{z_1,z_2}(a,b)&= \llangle\frac{a}{z_1-\psi_1},\frac{b}{z_2-\psi_2}\rrangle_{0,2}^{\cF(m,r),T}.
\end{align*}
We have the following identity:
\[
    V_{z_1,z_2}(a,b) = \frac{1}{z_1+z_2}\sum_{\alpha\in\{0,\dots,m+r-1\}}(a,\cS(\hat{\phi}_\alpha(\bq))(z_{1}))\cdot(b,\cS(\hat{\phi}_\alpha(\bq))(z_{2})).
\]
In the following context, we write $\cS(a)$ to represent $\cS(a)(z)$ for short. 

The $T$-equivariant $J$-function is characterized by
$$
(J^{T}(z),a)_{\cF(m,r),T} = (1,\cS(a)\big|_{Q=1})_{\cF(m,r),T}
$$
for any $a\in H_{\CR,T}^*(\cF(m,r))$.

We consider several different (flat) bases for $H_{\CR,T}^*(\cF(m,r))$:
\begin{enumerate}
\item The classical canonical basis: $\phi_\alpha:=\phi_{\alpha}(\bq)_{\bq\to0}$
\item The basis dual to the classical canonical basis with respect to the $T$-equivariant Poincar\'e pairing:
$\phi^\alpha =\Delta^{\alpha}(0)\phi_\alpha$.
\item The classical normalized canonical basis $\hat{\phi}_\alpha=\Delta^{\alpha}(0)^{\frac{1}{2}}\phi_\alpha$.
\item The quantum canonical basis $\phi_{\underline{\alpha}}:=\phi_{\alpha}(\bq)$.
\item The quantum normalized canonical basis $\phi_{\hat{\underline{\alpha}}}:= \hat{\phi}_\alpha(\bq)$.
\end{enumerate}

For $\alpha,\beta\in \{0,1,\dots,m+r-1\}$, define:
\begin{align*}
  S^{\ualpha}_{\ \ubeta}(z)&= (\phi^\alpha(\bq), \cS(\phi_\beta(\bq))),\quad
  S^{\hat\ualpha}_{\ \ubeta}(z) = (\hat\phi_\alpha(\bq), \cS(\phi_\beta(\bq))),\\
  S^{\ualpha}_{\ \hat\ubeta}(z) &= (\phi^\alpha(\bq), \cS(\hat\phi_\beta(\bq))), \quad
  S^{\hat\ualpha}_{\ \hat\ubeta}(z) = (\hat\phi_\alpha(\bq), \cS(\hat\phi_\beta(\bq))).
\end{align*}

Let $U=\mathbf{diag}(u^0,\dots, u^{m+r-1})$.
By Givental's theorem \cite{Giv98b}, there exists a unitary $R(z)$ (i.e. $R(z)R^T(-z)=I$) such that $S=\Psi R(z)e^{U/z}$ is a fundamental solution of the quantum differential equation, with $R(z) =\one + R_1z + R_2 z^2 + \dots$ a formal power series in $z$. Furthermore, $R(z)$ is unique up to 
a right multiplication of $\exp(a_1 z+a_3z^3+a_5z^5+\dots)$, where $a_i$ are complex diagonal matrices. More results about the A-model $R$-matrix of $\cF(m,r)$ could be found in \cite{Tang,Lan25}.

\subsection{Mirror theorem of $\cF(m,r)$}
We follow the results and notations in \cite{Lan25} on the equivariant $I$-function.

We introduce a basis of $H^*_{\CR,T}(\cF(m,r))$,
\begin{equation*}
    \begin{aligned}
\one_{h,r}=\one_{h} , \ 1\leq h\leq r-1&, \quad \one_{0,r}=(H+\bp_{+})/\bp,\\
\one_{-h,-m}=\one_{-h}, \ 1\leq h\leq m-1&
,\quad \one_{0,-m}=-(H+\bp_{-})/\bp.
    \end{aligned}
\end{equation*}

Let $\bold{q}^{d}=q_{*}^{d_{*}}\dprod_{i=-m+1}^{r-1}q_{i}^{d_{i}}$, $\deg q_{*}=\frac{m+r}{mr}$, $\deg q_{i}=\frac{r-i}{r}$ for $i\geq 0$, and $\deg q_{i}=\frac{m+i}{m}$ for $i<0$. Let $\tilde{d}_{-m}=\frac{1}{m}(d_{*}-\dsum_{i=1}^{m-1}id_{-i})$, $\tilde{d}_{r}=\frac{1}{r}(d_{*}-\dsum_{i=1}^{r-1}id_{i})$.
In this convention, let $$\deg d := \deg \bq^d = \tilde{d}_{-m} + \tilde{d}_r + \sum_{i=-m+1}^{r-1} d_i.$$

Let $\bK_\eff$ be the $K$-effective class of $\cF(m,r)$.
By \cite{Lan25}, $\bK_{\eff}=\bK_{\eff,r}\cup\bK_{\eff,-m}$ is as following:
\begin{equation*}
    \begin{aligned}
        \bK_{\eff,r}&=\{(d_{-m+1},\dots,d_{r-1},d_{*}):d_{i}, d_*\in\bZ_{\geq0},\tilde{d}_{-m}\in \bZ_{\geq0}\},\\
        \bK_{\eff,-m}&=\{(d_{-m+1},\dots,d_{r-1},d_{*}):d_{i},d_*\in\bZ_{\geq0},\tilde{d}_{r}\in \bZ_{\geq0}\}.
    \end{aligned}
\end{equation*}

The $T$-equivariant $I$-function is as following:
\begin{equation*}
  \begin{aligned}
    I^T(\bq,z) = e^{(-\sum_{a\in I_0}{\bf w}_a\log q_a)/z}\Big(
      \sum_{d\in \bK_{\eff,r}}e^{\frac{-{\bp}_-\log q_*}{z}}\bq^dz^{-\lceil\deg d\rceil}\cdot \frac{\Ga(\frac{\bp}{rz}+1-\{-\tilde{d}_r\})}{\tilde{d}_{-m}!\prod_{i\in I_0}d_i!\Ga(\frac{\bp}{rz}+\tilde{d}_r+1)}\one_{v(d),r}
    \\ +
\sum_{d\in \bK_{\eff,-m}}e^{\frac{-{\bp}_+\log q_*}{z}}\bq^dz^{-\lceil\deg d\rceil}\cdot \frac{\Ga(\frac{-\bp}{mz}+1-\{-\tilde{d}_{-m}\})}{\tilde{d}_{r}!\prod_{i\in I_0}d_i!\Ga(\frac{-\bp}{mz}+\tilde{d}_{-m}+1)}\one_{v(d),-m}
    \Big),
  \end{aligned}
\end{equation*}
where $I_0= [-m+1,r-1]\cap \bZ$ and 
$v(d) = -m\{-\tilde{d}_{-m}\}+r\{-\tilde{d}_r\}$.

We write
\begin{equation}\label{eqn:I-function}
      I^T(\bq,z) = \sum_{0\leq h\leq r-1} I_{+,h}(\bq,z)\one_{h,r} + \sum_{-m+1\leq h\leq 0} I_{-,h}(\bq,z)\one_{h,-m}.
\end{equation}
Then we have
\[
  I_{+,h}(\bq,z) = e^{(-\sum_{a\in I_0}{\bf w}_a\log q_a)/z}\sum_{d\in \bK_{\eff,r}\atop{r\{-\tilde{d}_r\}=h}}e^{\frac{-\bp_{-}\log q_*}{z}}\bq^dz^{-\lceil\deg d\rceil}\cdot \frac{\Ga(\frac{\bp}{rz}+1-\{-\tilde{d}_r\})}{\tilde{d}_{-m}!\prod_{i\in I_0}d_i!\Ga(\frac{\bp}{rz}+\tilde{d}_r+1)},
\]
\[
  I_{-,h}(\bq,z) = e^{(-\sum_{a\in I_0}{\bf w}_a\log q_a)/z}\sum_{d\in \bK_{\eff,-m}\atop{\{-\tilde{d}_{-m}\}=-\frac{h}{m}}}e^{\frac{-\bp_{+}\log q_*}{z}}\bq^dz^{-\lceil\deg d\rceil}\cdot \frac{\Ga(\frac{-\bp}{mz}+1-\{-\tilde{d}_{-m}\})}{\tilde{d}_{r}!\prod_{i\in I_0}d_i!\Ga(\frac{-\bp}{mz}+\tilde{d}_{-m}+1)}.
\]
By the equivariant mirror theorem \cite{CCIT15,CCK}, we have:
\begin{equation*}
    I^{T}(\bq,z)=J^{T}(\tau(\bq),z)
\end{equation*}
where $\tau$ is the equivariant mirror map defined in \cite{CCIT15,CCK}. The explicit result of equivariant mirror map for $\cF(m,r)$ is calculated in \cite{Lan25}.

\subsection{The graph sum formula for descendant Gromov-Witten potential}\label{sec:Givental-graph}
In this subsection, we introduce the graph sum formula for the generating functions of descendant
Gromov-Witten invariants. We introduce a labeled graph $\vec{\Ga}=(\Ga,g,\beta,k)$ as follows. Here $\Ga$ is a connected graph equipped with the following data: 
\begin{enumerate}
    \item $V(\Gamma)$ is the set of vertices in $\Gamma$.
    \item $H(\Gamma)$ is the set of half-edges in $\Gamma$. $H(\Ga)$ is equipped with a vertex assignment $v: H(\Ga)\rightarrow V(\Ga)$ and an involution $\iota$.
    \item $E(\Gamma)$ is the set of edges in $\Gamma$, which is defined by the 2-cycles of $\iota$ in $H(\Ga)$.
    \item $L(\Gamma)$ is the set of leaves in $\Gamma$, which is defined by the fixed points of $\iota$.
    It admits a disjoint splitting $L(\Ga)=L^o(\Ga)\sqcup L^1(\Ga)$, 
    where $L^o(\Ga)=\{l_1,\dots,l_n\}$ is the set of ordinary leaves in $\Ga$, with $n=|L^o(\Ga)|$, and
    $L^1(\Ga)$ is the set of dilaton leaves in $\Ga$. 
\end{enumerate}

With the above notations, we introduce the following labels:
\begin{itemize}
  \item Genus $g: V(\Gamma)\to \mathbb{Z}_{\geq 0}$;
  \item Marking $\beta: V(\Gamma)\to \{0,\cdots, m+r-1 \}$;
  \item Height $k: H(\Gamma)\to \mathbb{Z}_{\geq 0}$. Moreover, we require $k(l)\geq 2$ for every dilaton leaf $l\in L^1(\Ga)$.
\end{itemize}

An automorphism of the labeled graph $\vec{\Ga}=(\Ga,g,\beta,k)$ is an automorphism of the underlying graph $\Ga$
preserving all incidence relations and all decorations $(g,\beta,k)$. Moreover, it fixes each ordinary leaf $l_i\in L^o(\Ga)$. The dilaton leaves are 
unlabeled and may be permuted by automorphisms. We denote the automorphism group by $\Aut(\vec{\Ga})$.

Given an edge $e$, let $h_1(e)$ and $h_2(e)$ be the two half-edges associated to $e$. 
Note that the order of two vertices attached to an edge does not affect the graph sum formula in this paper.
Note that the marking on $V(\Gamma)$ induces a marking on $L(\Gamma)=L^o(\Gamma)\cup L^1(\Gamma)$ by $\beta(\ell)=\beta(v)$ where $\ell$ is attached to $v$.
Let $H(v)$ be the set of all half edges attached to $v$. Define the valence of $v\in V(\Gamma)$ as $\mathrm{val}(v)=|H(v)|$.
We say a labeled graph $\vec{\Gamma}=(\Gamma,g,\beta,k)$ is stable if $$2g(v)-2+\mathrm{val}(v)> 0,$$ for all $v\in V(\Gamma)$.

Let ${\bf\Gamma}(\cF(m,r))$ denote the set of all stable labeled graphs $\vec{\Gamma}=(\Gamma, g,\beta,k)$.
The genus of a stable labeled graph $\vec{\Gamma}$ is defined to be
\[
    g(\vec{\Gamma}) := \sum_{v\in V(\Ga)} g(v) + |E(\Gamma)| - |V(\Ga)| + 1
    = \sum_{v\in V(\Ga)}(g(v)-1)+ (\sum_{e\in E(\Ga)}1) + 1.
\]
Define 
\[
    {\bf \Ga}_{g,n}(\cF(m,r)) = \{\vec{\Ga}=(\Ga,g,\beta,k)\in {\bf\Ga}(\cF(m,r)) : 
    g(\vec{\Ga})=g, |L^o(\Ga)| = n\}.
\]

We assign weights to leaves, edges, and vertices of a labeled graph $\vec{\Gamma}\in {\bf \Ga}(\cF(m,r))$ as follows.
\begin{enumerate}
    \item \emph{Ordinary leaves}. To each ordinary leaf $l_j \in L^o(\Gamma)$ with $\beta(l_j)=\beta\in\{0,\dots,m+r-1\}$
        and $k(l_j)=k\in\bZ_{\geq 0}$, we assign the following descendant weight:
        \[
            (\cL^{\bf u})^\beta_k(l_j) = [z^k](\sum_{\alpha,\gamma \in \{0,\dots,m+r-1\}}\left(
                \frac{{\bf u}^\alpha_j(z)}{\sqrt{\Delta^\alpha(q)}}S^{\hat{\underline{\gamma}}}_{\ \hat{\underline{\alpha}}}(z)
            \right)_+ R(-z)_\gamma^{\ \beta}).
        \]
    where $(\cdot)_+$ means taking the nonnegative powers of $z$.
    \item \emph{Dilaton leaves}. To each dilaton leaf $l\in L^1(\Ga)$ with $\beta(l)=\beta \in \{0,\dots,m+r-1\}$ and 
        $2\leq k(l)=k\in\bZ_{\geq 0}$, we assign
        \[
            (\cL^1)^\beta_k(l) = [z^{k-1}]\Big(-\sum_{\alpha\in \{0,\dots,m+r-1\}}\frac{1}{\sqrt{\Delta^\alpha(q)}}R_\alpha^{\ \beta}(-z)\Big).
        \] 
    \item \emph{Edges}. To an edge connecting a vertex marked by $\alpha\in\{0,\dots,m+r-1\}$ and a vertex marked by $\beta\in\{0,\dots,m+r-1\}$,
    and with heights $k$ and $l$ at the corresponding half-edges, we assign
        \[
            \cE^{\alpha,\beta}_{k,l} = [z^kw^l]\Big(
                \frac{1}{z+w}(\delta_{\alpha,\beta}-\sum_{\gamma\in \{0,\dots,m+r-1\}}R_\gamma^{\ \alpha}(-z)R_\gamma^{\ \beta}(-w))
            \Big).
        \]
    \item \emph{Vertices}. To a vertex $v$ with genus $g(v)=g\in \bZ_{\geq 0}$ and with marking $\beta(v)=\beta$,
        with $m=\val(v)$ half-edges attached to it with heights $k_1,\dots,k_m\in\bZ_{\geq 0}$, we assign
        \[
            \Big(\sqrt{\Delta^\beta(q)}\Big)^{2g(v)-2+\val(v)}\langle\tau_{k_1}\dots\tau_{k_{m}}\rangle_g,
        \]
        where $\langle\tau_{k_1}\dots\tau_{k_{m}}\rangle_g = \int_{\overline{\cM}_{g,m}}\psi_1^{k_1}\dots\psi_{m}^{k_{m}}$.
\end{enumerate}
We define the weight of a labeled graph $\vec{\Ga}\in {\bf\Ga}_{g,n}(\cF(m,r))$ to be 
\begin{equation*}
    \begin{aligned}
        \omega^{\bf u}_A(\vec{\Ga}) = &\prod_{v\in V(\Ga)}\Big(\sqrt{\Delta^{\beta(v)}(q)}\Big)^{2g(v)-2+\val(v)}\langle\prod_{h\in H(v)}\tau_{k(h)}\rangle_{g(v)}
        \prod_{e\in E(\Ga)}\cE^{\beta(v_1(e)),\beta(v_2(e))}_{k(h_1(e)),k(h_2(e))}
        \\
        & \cdot \prod_{l\in L^1(\Ga)}(\cL^1)^{\beta(l)}_{k(l)}(l)\prod_{j=1}^n(\cL^{\bf u})^{\beta(l_j)}_{k(l_j)}(l_j).
    \end{aligned}
\end{equation*}
With the above definition of the weight of a labeled graph, we have the following theorem which expresses the $T$-equivariant descendant
Gromov-Witten potential of $\cF(m,r)$ in terms of graph sum.
\begin{theorem}[\cite{Giv98b,Zong15}]\label{thm:descendant-graph-sum}
    Suppose that $2g-2+n > 0$. Then
    \[
        \llangle{\bf u}_1, \dots, \bu_n\rrangle_{g,n}^{\cF(m,r),T} = \sum_{\vec{\Ga}\in{\bf \Ga}_{g,n}(\cF(m,r))}
        \frac{\omega^\bu_A(\vec{\Ga})}{|\Aut(\vec{\Ga})|}.
    \]
\end{theorem}

\section{Equivariant open Gromov-Witten theory of $(\cF(m,r),S^1)$}\label{sec:openGW}
\subsection{Moduli spaces of stable maps to ($\cF(m,r),S^{1}$)}

Then we have
\[
    H_2(\cF(m,r),S^{1}) = \bZ[D_1]\oplus \bZ[D_2].
\]
Sometimes, we identify the relative homology group $H_2(\cF(m,r),S^1)$ to $\bZ^2$ so that the generators $(1,0)$ and $(0,1)$ represent $D_1$ and $D_2$ 
respectively. 
The relative homology group $H_2(\cF(m,r),S^1)$ could also be viewed as the extension of the homology groups $H_2(\cF(m,r))$ and $H_1(S^1)$.
Consider the short exact sequence
\begin{equation}
    \begin{tikzcd}\label{eqn:relative-homology}
        0 \arrow[r] & H_2(\cF(m,r))= \bZ[\cF(m,r)] \arrow[r, "\delta"] & H_2(\cF(m,r),S^1) \arrow[r, "\partial"] & H_1(S^1)= \bZ[S^1] \arrow[r] & 0.
    \end{tikzcd}
\end{equation}
For $\beta=d[\cF(m,r)]\in H_2(\cF(m,r))$, the map $\delta:H_2(\cF(m,r))\rightarrow H_2(\cF(m,r),S^1)$ is defined by $\delta(d[\cF(m,r)])=(d,d)$.
The connecting map $\partial:H_2(\cF(m,r),S^1)\rightarrow H_1(S^1)$ is given by $\partial(a,b) = (b-a)[S^1]$. 

Let $(\Si,\partial\Si)$ be a prestable bordered Riemann surface, and let $\partial\Si=R_1\cup\dots\cup R_h$.
The (small) genus $g(\Si)$ of $\Si$ is the genus of the closed Riemann surface obtained by capping each boundary component of $\Si$ with a disk.
A genus-$g$ bordered Riemann surface with $h$ boundary components is called a Riemann surface of topological type $(g,h)$.

We are concerned with the maps $u:(\Si,\partial\Si)\rightarrow
(\cF(m,r),S^1)$.
The \emph{degree} of the map $u$ is the element 
\[
    \beta' = (d_-,d_+)= u_*[\Si]\in H_2(\cF(m,r),S^1).
\]
Let
\[
\mu_i[S^1]= (u\mid_{R_i})_*[R_i]\in H_1(S^1)=\bZ[S^1],\quad i=1,\cdots,h.
\]
The number $\mu_i$ is called the $i$-th winding number. Let $\beta'= u_*[\Si]\in H_2(\cF(m,r),S^1)$, $\vec{\mu} = (\mu_1,\dots,\mu_{h})\in\bZ_{\neq 0}^h$ and let $J_\pm = \{j\in \{1,\dots,h\}:\pm \mu_j > 0\}$.
Then there exists $d\in\bZ_{\geq 0}$ such that 
$$\beta' = d[\cF(m,r)] - \sum_{j\in J_-} \mu_j[D_1] + \sum_{j\in J_+}\mu_j[D_2].$$
By the short exact sequence \eqref{eqn:relative-homology}, we have
\[
        d_+-d_- = \sum_{j=1}^{h}\mu_j,
\]
\[
        d = d_+ - \sum_{j\in J_+}\mu_j = d_- +\sum_{j\in J_-}\mu_j.
\]

Let $(\Si,\partial\Si,x_1,\dots,x_n)$ be a prestable bordered Riemann surface with $n$ interior marked points,
and consider stable maps:
\[
    u: (\Si,\partial\Si,x_1,\dots,x_n)\rightarrow (\cF(m,r),S^1).
\]
Here, a stable map is a map whose automorphism group is finite.

Let $\overline{\cM}_{(g,h),n}(\cF(m,r),S^1 | \beta',\vec{\mu})$ be the moduli space of degree $\beta'$ stable maps to $(\cF(m,r),S^1)$ from type $(g,h)$ bordered Riemann surfaces with $n$ interior marked points, such that
the winding numbers are given by $\mu_i\in\bZ$.
\subsection{Equivariant open Gromov-Witten invariants}
Let $\gamma_1,\dots,\gamma_n\in H^*_{\CR,T}(\cF(m,r);\bC)$, $\beta'\in H_2(\cF(m,r),S^{1})$.
The $T$-action on $\cF(m,r)$ induces a $T$-action on the moduli space $\overline{\cM}_{(g,h), n}(\cF(m,r),S^{1}|\beta',\vec{\mu})$.
Let $F:=\overline{\cM}_{(g,h), n}(\cF(m,r),S^{1}|\beta',\vec{\mu})^T$ be the $T$-fixed locus. The evaluation map $\ev_i: \overline{\cM}_{(g,h), n}(\cF(m,r),S^{1}|\beta',\vec{\mu})\to \cI\cF(m,r)$ is $T$-equivariant.

For $i=1,\dots,n$, let $\bL_i$ be the $i$-th tautological line bundle over $\overline{\cM}_{(g,h), n}(\cF(m,r),S^{1}|\beta',\vec{\mu})$ formed
by the cotangent line at the $i$-th marked point. Define the $i$-th descendant class $\psi_i$ as
$$
\psi_i := c_1(\bL_i)\in H^2(\overline{\cM}_{(g,h), n}(\cF(m,r),S^{1}|\beta',\vec{\mu});\bQ).
$$
We choose a $T$-equivariant
lift $\psi_i^{T}\in H^2_{T}(\overline{\cM}_{(g,h), n}(\cF(m,r),S^{1}|\beta',\vec{\mu});\bQ)$
of $\psi_i$. 
We define the $T$-equivariant open Gromov-Witten invariants of $(\cF(m,r),S^{1})$ \cite{Liu20}:
\begin{equation}\label{eqn:openGW}
    \langle \tau_{a_1}(\gamma_1)\dots\tau_{a_n}(\gamma_n)\rangle^{\OGW}_{g,\beta',\vec{\mu}}
    := \int_{[F]^\vir}
    \frac{\prod_{i=1}^n (\psi_i^T)^{a_i}\ev_i^*(\gamma_i)|_F}{e_T(N^\vir)} \in \bC(\bw),
\end{equation}
where $[F]^\vir$ is the virtual fundamental class of $F$, and $N^\vir$ is the virtual normal bundle of $F$.  Since $F$ is a compact orbifold without boundary, the above integral is well-defined.

We also define the following double correlator:
\[
    \llangle \tau_{a_1}(\gamma_1),\dots,\tau_{a_n}(\gamma_n)\rrangle_{g,\vec{\mu}}^{\OGW} := 
    \sum_{\beta\in E(\cF(m,r))}\sum_{l=0}^\infty \frac{Q^\beta}{l!}\langle \tau_{a_1}(\gamma_1),\dots,\tau_{a_n}(\gamma_n)
    ,\bt^l\rangle_{g,\beta',\vec{\mu}}^{\OGW}.
\]

\begin{remark} \rm
    There is a definition of equivariant open Gromov-Witten invariants for 3-dimensional Calabi-Yau smooth toric Deligne-Mumford stacks in \cite[Section 3.6]{FLT}.
In \cite{FLT}, the virtual normal bundle $N^\vir$ is defined via the tangent-obstruction complex. Each connected component of the fixed locus $F$ is a compact orbifold without boundary and 
the virtual fundamental class $[F]^\vir = [F]$. The equivariant open Gromov-Witten invariants are defined by integrating the inverse of the Euler class of $N^\vir$ (denoted by $e_T(N^\vir)^{-1}$) over $[F]^\vir$ (see \cite[Equation (13)]{FLT}).
In the case of $(\cF(m,r),S^1)$, the equivariant open Gromov-Witten invariants \emph{via integration over the fixed locus} can be defined similarly. After the computation of $e_T(N^\vir)^{-1}$, we obtain the right hand side of the expression in Proposition \ref{prop:localization}.
\end{remark}

\subsection{Disk factor as relative GW invariants}
For $r\geq 1$, let $\bP^1[r]$ be the projective line with a single stacky point of order $r$ at $0$. Let $\mu\in\bZ_{>0}$ and let $\gamma=\bar\mu$ be the representative of $\mu$ in $\bZ_r$. Let $\overline{\cM}_{0,\gamma}(\bP^1[r],\mu)$ be the moduli space of relative map with relative condition $\mu$ over $\infty$. 


The standard $\bC^*$-action defined by $t\cdot [z_0,z_1] = [z_0,tz_1]$, lifts canonically to $\bC^*$-actions on $\bP^1[r]$ and $\overline{\cM}_{0,\gamma}(\bP^1[r],\mu)$. The tangent weights at $[0/\bZ_r]$, $\infty\in\bP^1[r]$ are $\bp/r$ and $-\bp$. The $\bC^*$-fixed locus of $\overline{\cM}_{0,\gamma}(\bP^1[r],\mu)$ is identified with 
$$
    \overline{\cM}_0^{\bC^*}\cong \overline{\cM}_{0,(\mu,-\mu)}(\cB \bZ_r)\times_{\bar{I}\cB \bZ_r} P,
$$
where $\bar{I}\cB \bZ_r$ is the rigidified inertia stack of $\cB \bZ_r$ and $P$ is the moduli stack of $\bC^*$-fixed Galois covers of degree $\mu$.
We define
$$
    \langle \bold{1}_{\mu,r}\rangle^{\OGW}_{0,\mu{[S^{1}]},\mu} := \int_{\overline{\cM}_0^{\bC^*}}\frac{1}{e(N^\vir)},
$$
where $e(N^\vir)$ is the equivariant Euler class of the virtual normal bundle. We cite \cite[Equation (16)]{JPT11} for details. 

Similarly, we define the disk invariants $ \langle \bold{1}_{\mu,-m}\rangle^{\OGW}_{0,\mu{[S^{1}]},\mu}$ for $\mu\in\bZ_{<0}$.

For $\mu<0$, we define $D^{1}(\mu)$ such that
$$
    \langle \one_{\mu,-m}\rangle^{\OGW}_{0,\mu{[S^{1}]},\mu}
    = D^1(-\mu)(\frac{\mu}{\bp})\cdot(\one_{\mu,-m},\one_{-\mu,-m})_{\cF(m,r),T},
$$
For $\mu>0$, we define $D^{2}(\mu)$ such that
$$
    \langle \bold{1}_{\mu,r}\rangle^{\OGW}_{0,\mu{[S^{1}]},\mu}
    = D^2(\mu)(\frac{\mu}{\bp})\cdot(\one_{\mu,r},\one_{-\mu,r})_{\cF(m,r),T}.
$$
We conclude that
\begin{equation*}\label{eqn:disk-factor}
    \begin{aligned}
        &D^1(\mu) = \Big(-\frac{m}{\bp}\Big)^{1-\delta_{0,\langle\frac{\mu}{m}\rangle}}(-1)^{\lfloor\frac{\mu}{m}\rfloor-1}\frac{\mu^{\lfloor\frac{\mu}{m}\rfloor-2}}{\lfloor\frac{\mu}{m}\rfloor!\bp^{\lfloor\frac{\mu}{m}\rfloor-2}},
        \\
        &D^2(\mu) = \Big(\frac{r}{\bp}\Big)^{1-\delta_{0,\langle\frac{\mu}{r}\rangle}}\frac{\mu^{\lfloor\frac{\mu}{r}\rfloor-2}}{\lfloor\frac{\mu}{r}\rfloor!\bp^{\lfloor\frac{\mu}{r}\rfloor-2}}.
    \end{aligned}
\end{equation*}

\subsection{Virtual localization formula}
\begin{defn}[Decorated graphs (a)]
    Let $n\in\bZ_{\geq 0}$, $\vec{j}\in (\mathrm{Box}(\Sigma))^{n}$ and $\beta' = d[\cF(m,r)] - \sum_{j\in J_-} \mu_j[D_1] + \sum_{j\in J_+}\mu_j[D_2]\in H_2(\cF(m,r),S^1)$. A genus $g$,$n$-pointed, $\vec{j}$ -twisted, 
    degree $\beta'$ decorated graph for $(\cF(m,r),S^{1})$ is a tuple $\widetilde{\Ga}=(\Ga,\vec{f},\vec{d}, \vec{g},\vec{s},\vec{k})$ consisting of 
    the following data.
    \begin{enumerate}
        \item $\Gamma$ is a compact connected graph. Let $V(\Ga)\sqcup V_\circ(\Ga)$ denote
        the set of vertices in $\Ga$, denoted by $\bullet$ and $\circ$, respectively. The $\circ$ vertex is univalent.
        Let $E(\Ga)$ denote the set of edges, where an edge $e$ is a line connecting two $\bullet$ vertices.
        Let $L(\Ga) = \{l_j: j=1,\dots,h\}$ denote the set of leaves, where a leaf $l_j$ is a line connecting one $\bullet$ vertex and one $\circ$ vertex.
        We require the leaves to be ordered. 
        Let $F(\Ga)$ be the set of flags:
        \[
            \{(e,v)\in E(\Ga)\times V(\Ga):v\in e\} \cup \{(l_j,v)\in L(\Ga)\times V(\Ga):v\in l_j\}.
        \]
        For each $v\in V(\Ga)$, let $F_v$ (\emph{resp.} $E_v$, $L_v$) denote the flags (\emph{resp.} edges, leaves) attached to $v$, and let $\val(v)=|F_v|$ denote the number of flags incident to $v$.
        \item The \emph{label map} $\vec{f}: V(\Ga)\rightarrow \{-m,r\}$ labels each $\bullet$ with a number.
            If $v_1,v_2\in V(\Ga)$ are connected by an edge, we require
            $\vec{f}(v_1)\neq \vec{f}(v_2)$.
        \item The \emph{degree map} $\vec{d}:E(\Ga)\cup L(\Ga)\rightarrow \bZ_{>0}$ sends an edge $e$ (\emph{resp.} a leaf $l_j$) to a positive integer $\vec{d}(e)=d_e$ (\emph{resp.} $\vec{d}(l_j)=d_{l_j}$).
        \item The \emph{genus map} $\vec{g}:V(\Ga)\rightarrow\bZ_{\geq 0}$ sends a vertex $v\in V(\Ga)$ to a nonnegative integer $g_v$.
        \item The \emph{marking map} $\vec{s}:\{1,2,\dots,n\}\rightarrow V(\Ga)$. For each $v\in V(\Ga)$, define $S_v := \vec{s}^{-1}(v)$, and $n_v=|S_v|$.
    \item $\vec{k}$ is the twisting map that sends each flag $(e,v), (l_j,v)\in F(\Gamma)$ to some $k_{(e,v)}, k_{(l_j,v)}\in G_{\vec{f}(v)}$, and each marking $i\in\{1,\dots, n\}$ to some $k_{i}\in G_{\vec{f}(\vec{s}(i))}$.
    \end{enumerate}
    The data are required to satisfy the following conditions:
    \begin{itemize}
        \item [(i)] (genus) $g = \sum_{v\in V(\Ga)}g_v + |E(\Ga)| - |V(\Ga)| + 1$.
        \item [(ii)] (degree) $d = \sum_{e\in E(\Ga)}d_e$.
        \item [(iii)] (winding numbers) Let $l_j\in L(\Ga)$ be the $j$-th leaf, and let $v_j\in V(\Ga)$ be its incident $\bullet$ vertex, then 
                $\mu_j = \mathrm{sgn}(\vec{f}(v_{j}))\vec{d}(l_j)$.
        \item [(iv)] (compatibility along an edge) For any edge $e\in E(\Ga)$, if $v, v'\in V(\Ga)$ are the two incident vertices, then $k_{(e,v)} \in G_{\vec{f}(v)}$ and $k_{(e,v')}\in G_{\vec{f}(v')}$ are determined by $\vec{d}(e)\in\bZ_{>0}$. For any leaf $l_j\in L(\Ga)$, if $v\in V(\Ga)$ is the incident vertex, then $k_{(l_j,v)}\in G_{\vec{f}(v)}$ is determined by $\vec{d}(l_j)\in \bZ_{>0}$.
        \item[(v)] (compatibility at a vertex) Given $v\in V(\Ga)$, let $E_v, L_v$ and $S_v$ be the set of edges, leaves, and marked points attached to $v$, then
        \[
            \prod_{e\in E_v}k^{-1}_{(e,v)}\prod_{l_j\in L_v}k^{-1}_{(l_j,v)}\prod_{i\in S_v}k_i = 1.
        \]
        \item [(vi)](compatibility with $\vec{j}$) For each $i=1,\dots, n$, the pair $(p_{\vec{f}(\vec{s}(i))},k_{i})$ represents a point in
the inertia component $\cF_{j_{i}}$.
    \end{itemize}
    Let $G_{g,\vec{j}}(\cF(m,r),S^{1}|\beta',\vec{\mu})$ be the set of all decorated graphs $\vec{\Ga}=(\Ga,\vec{f},\vec{d}, \vec{g},\vec{s},\vec{k})$ satisfying the above constraints. 
\end{defn}

Given a decorated graph $\vec{\Ga}\in G_{g,\vec{j}}(\cF(m,r),S^{1}|\beta',\vec{\mu})$. Let $V^S(\Ga)$ be the set of stable vertices:
\[
    V^S(\Ga) := \{v\in V(\Ga):2g_v-2+\val(v)+n_v >0\}.
\]

\noindent
Let $\Aut(\widetilde{\Ga})$ denote the group of automorphisms of $\widetilde{\Ga}$, and let $A^0_{\widetilde{\Ga}}$ be the group of covering automorphisms
\[
    A^0_{\widetilde{\Ga}} = \prod_{e\in E(\Ga)}\bZ_{\vec{d}(e)}\times\prod_{l_j\in L(\Ga)}\bZ_{\vec{d}(l_j)}.
\]
\noindent
The automorphism group $A_{\widetilde{\Ga}}$ fits in the short exact sequence: 
\[
    1 \rightarrow A^0_{\widetilde{\Ga}}
    \rightarrow A_{\widetilde{\Ga}}\rightarrow \Aut(\widetilde{\Ga})\rightarrow 1.
\]
\noindent
We set 
\[
    \overline{\cM}_{\widetilde{\Ga}} := \prod_{v\in V^S(\Ga)}\overline{\cM}_{g_v,\vec{k}(v)}(\BG_{\vec{f}(v)}),
\quad
    F_{\widetilde{\Ga}} := [\overline{\cM}_{\widetilde{\Ga}}/A_{\widetilde{\Ga}}].
\]

Every decorated graph $\widetilde{\Ga}\in G_{g,\vec{j}}(\cF(m,r),S^{1}|\beta',\vec{\mu})$ represents a topological type of the $T$-invariant stable open map $u:(\Si,\partial\Si,x_1,\dots,x_n)\rightarrow (\cF(m,r),S^{1})$ in the following way:
\begin{enumerate}
    \item Every $v\in V(\Ga)$ represents a stable curve or a single point $\Si_v$ in the domain curve $\Si$. 
    If $\Si_v$ is a stable curve, it is of genus-$g_v$ with marked points $S_v$.
    The map $u$ contracts $\Si_v$ to the fixed-point $p_{\vec{f}(v)}$.
    \item Every leaf $l_j\in L(\Ga)$ represents a disk $D_j$ in the domain curve $\Si$. 
    The map $u|_{D_j}$ is a standard degree $\vec{d}(l_j)=|\mu_j|$ cover of disk branched at $p_{\vec{f}(v_j)}$, where $v_j$ is the $\bullet$ vertex attached at $l_j$.
    \item Every edge $e\in E(\Ga)$ represents a sphere $\cF(m,r)\cong C\subset \Si$,
    so that $u|_C$ is a degree $\vec{d}(e)=d_e$ cover of $\cF(m,r)$ branched at $p_{-m},p_{r}$. 
    \item For each $i\in\{ 1,..., n\}$, the pair $(p_{\vec{f}(\vec{s}(i))},k_{i})$ represents a point in the inertia component $\cF_{j_{i}}$.
\end{enumerate}

Vice versa, given a $T$-invariant stable open map $u$, one can read off its decorated graph $\widetilde{\Gamma}$. Therefore, we can use the decorated graphs to 
describe the connected components of the 
$T$-fixed locus of the moduli space $\overline{\cM}_{(g,h), \vec{j}}(\cF(m,r),S^{1}|\beta',\vec{\mu})$:
\begin{equation}\label{eq:fixed-locus}
    \overline{\cM}_{(g,h), \vec{j}}(\cF(m,r),S^{1}|\beta',\vec{\mu})^T = \bigcup_{\widetilde{\Ga} \in G_{g,\vec{j}}(\cF(m,r),S^{1}|\beta',\vec{\mu})}
    F_{\widetilde\Ga}.
\end{equation}

\begin{figure}[H]

\tikzset{every picture/.style={line width=0.75pt}} 

\begin{tikzpicture}[x=0.75pt,y=0.75pt,yscale=-1,xscale=1]

\draw    (153.19,60.5) -- (200.76,77) ;
\draw    (153.69,99) -- (200.76,78.18) ;
\draw    (153.19,158.5) -- (199.69,138) ;
\draw    (301.26,137.68) -- (339.26,158.18) ;
\draw    (199.76,77) .. controls (208.69,63.23) and (292.83,63.23) .. (301.76,77) ;
\draw    (199.76,77) .. controls (218.29,90.03) and (283.23,90.03) .. (301.76,77) ;
\draw    (199.69,136.5) .. controls (214.69,134.63) and (277.09,111.43) .. (301.76,77.68) ;
\draw    (199.69,138) .. controls (221.49,143.03) and (258.29,143.83) .. (299.26,138.68) ;

\draw   (429.4,58.01) .. controls (433.86,52.81) and (444.42,51.74) .. (453,55.62) .. controls (461.57,59.49) and (464.91,66.85) .. (460.45,72.05) .. controls (456,77.24) and (445.44,78.31) .. (436.86,74.44) .. controls (428.29,70.56) and (424.95,63.2) .. (429.4,58.01) -- cycle ;
\draw    (442.15,57.37) .. controls (436.67,62.88) and (443,68.28) .. (453.24,65.59) ;
\draw    (440.52,60.37) .. controls (445.63,60.42) and (448.48,62.76) .. (449.4,66.25) ;
\draw   (519.85,78.96) .. controls (511.15,74.11) and (508.81,63.7) .. (514.64,55.73) .. controls (520.46,47.76) and (532.23,45.23) .. (540.93,50.09) .. controls (549.63,54.95) and (551.96,65.35) .. (546.14,73.32) .. controls (540.32,81.3) and (528.55,83.82) .. (519.85,78.96) -- cycle ;
\draw    (524.79,64.23) .. controls (519.99,69.47) and (525.63,74.49) .. (534.67,71.87) ;
\draw    (523.37,67.07) .. controls (527.89,67.07) and (530.43,69.25) .. (531.28,72.53) ;
\draw    (528.36,52.77) .. controls (523.56,58.01) and (529.21,63.03) .. (538.24,60.41) ;
\draw    (526.95,55.61) .. controls (531.46,55.61) and (534,57.79) .. (534.85,61.07) ;
\draw   (421.44,115.03) .. controls (426.19,108.03) and (439.22,106.77) .. (450.55,112.21) .. controls (461.88,117.65) and (467.22,127.74) .. (462.47,134.74) .. controls (457.72,141.75) and (444.69,143.01) .. (433.36,137.57) .. controls (422.03,132.13) and (416.7,122.04) .. (421.44,115.03) -- cycle ;
\draw    (440.56,113.96) .. controls (436.79,118.25) and (441.23,122.36) .. (448.33,120.21) ;
\draw    (439.46,116.28) .. controls (443,116.28) and (445,118.07) .. (445.67,120.75) ;
\draw    (430.27,121.63) .. controls (426.5,125.91) and (430.93,130.02) .. (438.03,127.88) ;
\draw    (429.16,123.95) .. controls (432.71,123.95) and (434.7,125.74) .. (435.37,128.42) ;
\draw    (446.12,125.59) .. controls (442.35,129.88) and (446.79,133.99) .. (453.89,131.84) ;
\draw    (445.01,127.91) .. controls (448.56,127.91) and (450.56,129.7) .. (451.22,132.38) ;
\draw   (518.3,114.81) .. controls (521.48,108.73) and (529.83,106.99) .. (536.94,110.92) .. controls (544.06,114.84) and (547.26,122.95) .. (544.08,129.03) .. controls (540.9,135.1) and (532.55,136.85) .. (525.43,132.92) .. controls (518.31,129) and (515.12,120.89) .. (518.3,114.81) -- cycle ;
\draw    (527.56,113.15) .. controls (523.63,119.69) and (529.11,125.36) .. (537.06,121.71) ;
\draw    (526.51,116.62) .. controls (530.58,116.36) and (533.04,118.82) .. (534.05,122.69) ;
\draw   (453.16,113.03) .. controls (451.87,110.51) and (465.46,100.96) .. (483.52,91.69) .. controls (502.57,82.43) and (518.25,76.95) .. (519.55,79.47) .. controls (520.84,81.99) and (507.25,91.54) .. (488.19,101.81) .. controls (470.14,110.07) and (454.46,115.55) .. (453.16,113.03) -- cycle ;
\draw   (461.45,71.76) .. controls (461.45,68.74) and (472.06,66.29) .. (486.37,66.29) .. controls (500.68,66.29) and (512.29,68.74) .. (512.29,71.76) .. controls (512.29,74.78) and (500.68,77.23) .. (486.37,77.23) .. controls (472.06,77.23) and (461.45,74.78) .. (461.45,71.76) -- cycle ;
\draw   (455.43,55.45) .. controls (455.43,52.78) and (468.64,50.62) .. (484.93,50.62) .. controls (501.22,50.62) and (514.43,52.78) .. (514.43,55.45) .. controls (514.43,58.12) and (501.22,60.29) .. (484.93,60.29) .. controls (468.64,60.29) and (455.43,58.12) .. (455.43,55.45) -- cycle ;
\draw   (464.56,126.48) .. controls (464.35,123.2) and (476.01,119.8) .. (490.61,118.86) .. controls (505.21,117.93) and (517.21,119.82) .. (517.42,123.09) .. controls (517.63,126.36) and (505.97,129.77) .. (491.37,130.71) .. controls (476.77,131.64) and (464.76,129.75) .. (464.56,126.48) -- cycle ;
\draw    (452.22,76.49) .. controls (458.4,76.13) and (466.75,76.25) .. (470.95,83.34) ;
\draw    (452.22,76.09) .. controls (452.09,82.73) and (461.11,91.9) .. (464.04,91.16) ;
\draw   (470.84,83.21) .. controls (471.32,83.69) and (470.18,85.86) .. (468.3,88.06) .. controls (466.42,90.25) and (464.51,91.64) .. (464.03,91.16) .. controls (463.56,90.68) and (464.7,88.5) .. (466.58,86.31) .. controls (468.46,84.11) and (470.37,82.72) .. (470.84,83.21) -- cycle ;
\draw    (443.03,53.16) .. controls (444.02,47.06) and (445.43,38.83) .. (453.34,35.76) ;
\draw    (443.03,53.16) .. controls (449.57,54.31) and (460.02,46.81) .. (459.74,43.81) ;
\draw   (452.93,35.86) .. controls (453.48,35.46) and (455.45,36.92) .. (457.33,39.11) .. controls (459.21,41.31) and (460.29,43.41) .. (459.74,43.81) .. controls (459.19,44.21) and (457.22,42.75) .. (455.34,40.55) .. controls (453.46,38.36) and (452.38,36.25) .. (452.93,35.86) -- cycle ;
\draw    (455.82,139.69) .. controls (462,139.73) and (470.35,139.85) .. (474.37,146.64) ;
\draw    (455.82,139.69) .. controls (455.69,146.33) and (464.71,155.5) .. (467.64,154.76) ;
\draw   (474.44,146.81) .. controls (474.92,147.29) and (473.78,149.46) .. (471.9,151.66) .. controls (470.02,153.85) and (468.11,155.24) .. (467.63,154.76) .. controls (467.16,154.28) and (468.3,152.1) .. (470.18,149.91) .. controls (472.06,147.71) and (473.97,146.32) .. (474.44,146.81) -- cycle ;
\draw    (523.17,131.21) .. controls (521.19,137.07) and (518.47,144.96) .. (510.61,146.71) ;
\draw    (523.17,131.21) .. controls (516.9,129.01) and (505.37,134.72) .. (505.16,137.72) ;
\draw   (510.59,146.68) .. controls (509.98,146.98) and (508.27,145.22) .. (506.77,142.75) .. controls (505.27,140.28) and (504.55,138.03) .. (505.16,137.72) .. controls (505.77,137.42) and (507.47,139.18) .. (508.97,141.65) .. controls (510.47,144.12) and (511.19,146.37) .. (510.59,146.68) -- cycle ;
\draw    (485.8,60.29) .. controls (482.19,59.12) and (481.94,52.37) .. (485.8,50.62) ;
\draw    (486.37,77.23) .. controls (483.09,76.03) and (482.16,67.68) .. (486.37,66.29) ;
\draw    (487.09,102.03) .. controls (483.16,101.28) and (481.49,92.43) .. (486.93,90.62) ;
\draw    (489.08,130.33) .. controls (484.96,130.48) and (483.93,121.17) .. (488.2,118.95) ;
\draw    (485.8,50.62) .. controls (490.94,51.22) and (489.94,60.82) .. (485.8,60.29) ;
\draw    (486.37,66.29) .. controls (491.01,68.29) and (490.01,76.87) .. (486.37,77.23) ;
\draw    (486.93,90.62) .. controls (490.11,90.8) and (491.71,100) .. (486.93,102.03) ;
\draw    (488.2,118.95) .. controls (492.11,120.37) and (492.15,129.12) .. (489.08,130.33) ;
\draw   (444.89,196.15) .. controls (444.89,192.87) and (464.74,190.22) .. (489.23,190.22) .. controls (513.71,190.22) and (533.56,192.87) .. (533.56,196.15) .. controls (533.56,199.43) and (513.71,202.08) .. (489.23,202.08) .. controls (464.74,202.08) and (444.89,199.43) .. (444.89,196.15) -- cycle ;
\draw    (489.49,201.63) .. controls (484.89,200.13) and (484.69,191.83) .. (489.23,190.22) ;
\draw    (489.23,190.22) .. controls (493,191.03) and (493.8,199.93) .. (489.49,201.63) ;
\draw   (443.89,196.15) .. controls (443.89,195.87) and (444.12,195.65) .. (444.39,195.65) .. controls (444.67,195.65) and (444.89,195.87) .. (444.89,196.15) .. controls (444.89,196.43) and (444.67,196.65) .. (444.39,196.65) .. controls (444.12,196.65) and (443.89,196.43) .. (443.89,196.15) -- cycle ;
\draw   (533.56,196.15) .. controls (533.56,195.87) and (533.79,195.65) .. (534.06,195.65) .. controls (534.34,195.65) and (534.56,195.87) .. (534.56,196.15) .. controls (534.56,196.43) and (534.34,196.65) .. (534.06,196.65) .. controls (533.79,196.65) and (533.56,196.43) .. (533.56,196.15) -- cycle ;

\draw (146.33,55.4) node [anchor=north west][inner sep=0.75pt]  {$\circ $};
\draw (146.33,95.4) node [anchor=north west][inner sep=0.75pt]  {$\circ $};
\draw (146.33,155.4) node [anchor=north west][inner sep=0.75pt]  {$\circ $};
\draw (196.33,72.9) node [anchor=north west][inner sep=0.75pt]    {$\bullet $};
\draw (196.33,132.9) node [anchor=north west][inner sep=0.75pt]    {$\bullet $};
\draw (296.33,73.57) node [anchor=north west][inner sep=0.75pt]    {$\bullet $};
\draw (296.33,133.9) node [anchor=north west][inner sep=0.75pt]    {$\bullet $};
\draw (336.33,155.4) node [anchor=north west][inner sep=0.75pt]  {$\circ $};
\draw (176.33,216.06) node [anchor=north west][inner sep=0.75pt]    {$\widetilde{\Gamma} \in G_{g,\vec{j}}(\cF(m,r),S^{1}|\beta',\vec{\mu})$};
\draw (483.67,162.07) node [anchor=north west][inner sep=0.75pt]    {$\downarrow $};
\draw (426.34,190.06) node [anchor=north west][inner sep=0.75pt]    {$p_{r}$};
\draw (538.74,190.06) node [anchor=north west][inner sep=0.75pt]    {$p_{-m}$};
\draw (450.94,216.06) node [anchor=north west][inner sep=0.75pt]    {$\left(\cF(m,r) ,S^{1}\right)$};
\draw (170.74,150.06) node [anchor=north west][inner sep=0.75pt]    {$l_j$};
\end{tikzpicture}
\caption{The left figure is an example of the graph $\widetilde{\Gamma}$ $\in$ $G_{g,\vec{j}}(\cF(m,r),S^{1}|\beta',\vec{\mu})$. 
The right figure is an example of stable map corresponding to $\widetilde{\Gamma}$.}
\end{figure}
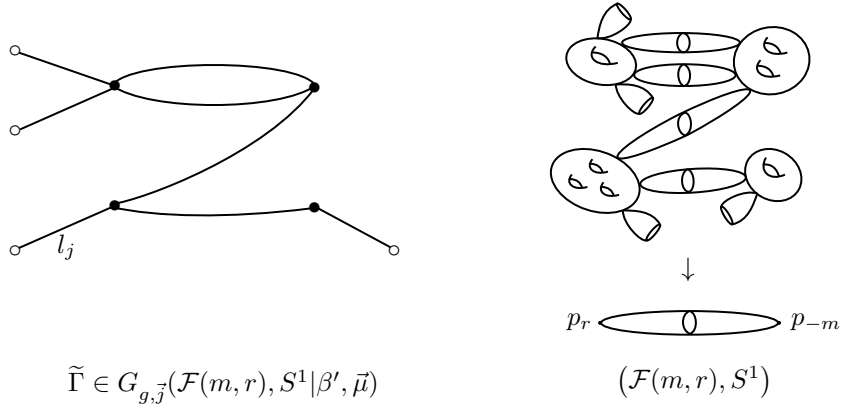

Given $\widetilde{\Ga}\in G_{g,\vec{j}}(\cF(m,r),S^{1}|\beta',\vec{\mu})$, we introduce the following notations:
\begin{itemize}
    \item (weight) We define
    \[
        {\bf w}(p_{-m}) = -\frac{\bp}{m}, \quad {\bf w}(p_{r}) = \frac{\bp}{r},
    \]
    For a flag $f\in F_v$, we define
    \begin{equation*}
        {\bf w}_f := \left\{
        \begin{aligned}
            &\frac{|G_{\vec{f}(v)}|{\bf w}(p_{\vec{f}(v)})}{d_e}, \quad f=(e,v),
            \\
            &\frac{|G_{\vec{f}(v)}|{\bf w}(p_{\vec{f}(v)})}{d_{l_j}}, \quad f=(l_j,v).    
        \end{aligned}
        \right.
    \end{equation*}
    \item (vertex contribution) Let $\phi_i\in G_i^*$ be the irreducible character which corresponds to the 1-dimensional $G_i$-representation of $T_{p_i}\cF(m,r)$, $i=-m,r$. Let $\bE_{\phi_i}$ be the $\phi_i$-twisted Hurwitz-Hodge bundle. For $v\in V(\Ga)$, with $i=\vec{f}(v)$, we define
        \[
            {\bf h}(v) = \frac{\Lambda^\vee_{\phi_i}({\bf w}(p_{\vec{f}(v)}))}{{\bf w}(p_{\vec{f}(v)})^{\delta_{|\widehat{G}_{v}|,1}}},
        \] 
        where $\Lambda^\vee_{\phi_i}(\su) := \sum_{j=0}^{\rank \bE_{\phi_i}} (-1)^j\lambda_j^{\phi_i}\su^{\rank \bE_{\phi_i}-j}$ and $\lambda_j^{\phi_i}$ is the $j$-th Hurwitz-Hodge class on $\overline{\cM}_{g,n}(\BG_{i})$. $\widehat{G}_{v}$ is the monodromy group of the $G_{\vec{f}(v)}$-cover of $\Sigma_{v}$ defined in \cite[Section 9.3.4]{Liu13}.
    \item (edge contribution) For $d\in \bZ_{>0}$, we define
        \[
            {\bf h}(e,d) = \frac{(-1)^{\lfloor\frac{d}{m}\rfloor} (\frac{d}{\bp})^{\lfloor\frac{d}{m}\rfloor+\lfloor\frac{d}{r}\rfloor}}{\lfloor\frac{d}{m}\rfloor!\lfloor\frac{d}{r}\rfloor!}.
        \]
    \item (node contribution) For $(e,v)\in E_v$, let $a_v := |G_{\vec{f}(v)}|$, we define 
        \[
            {\bf h}(e,v) = a_v{\bf w}(p_{\vec{f}(v)})^{\delta_{a_v|d_e}}.
        \]
\end{itemize}

By our definition of equivariant open Gromov-Witten invariants of $\cF(m,r)$ in \eqref{eqn:openGW}, we get the following proposition.
\begin{prop}\label{prop:localization} Let $\beta' = d[\cF(m,r)] - \sum_{j\in J_-} \mu_j[D_1] + \sum_{j\in J_+}\mu_j[D_2]\in H_2(\cF(m,r),S^{1})$. Then for $\gamma_1,\dots,\gamma_n\in H^*_{\CR,T}(\cF(m,r))$ and $g, a_1,\dots,a_n\in \bZ_{\geq 0}$, we have
    \begin{equation*}
        \begin{aligned}
            &\langle \tau_{a_1}(\gamma_1),\dots,\tau_{a_n}(\gamma_n)\rangle^{\OGW}_{g,\beta',\vec{\mu}} 
            \\
            = &\sum_{\vec{j}\in\mathrm{Box(\Sigma)}^n}\sum_{\widetilde{\Ga}\in G_{g,\vec{j}}(\cF(m,r),S^{1}|\beta',\vec{\mu})}\frac{1}{|\Aut(\widetilde{\Ga})|}\prod_{e\in E(\Ga)}\frac{{\bf h}(e,d_e)}{d_e}
            \prod_{v\in V(\Ga)}\Big(\prod_{(e,v)\in E_v} {\bf h}(e,v)\prod_{i\in S_v}i^*_{p_{\vec{f}(v)}}\gamma_{i}\Big)
            \\
            &\cdot\prod_{j\in J_-}D^1(-\mu_j)\prod_{j\in J_+}D^2(\mu_j)
            \prod_{v\in V(\Ga)}\int_{\overline{\cM}_{g_v,\vec{k}_{v}}(\BG_{\vec{f}(v)})}\frac{{\bf h}(v)\prod_{i\in S_v}\psi_i^{a_i}}{\prod_{l_j\in L_v}{\bf w}_{(l_j,v)}\prod_{f\in F_v}({\bf w}_{f}-\psi_{f})}.
        \end{aligned}
    \end{equation*}
    We use the following convention for the unstable integrals:
    \begin{equation*}
        \begin{aligned}
        \int_{\overline{\cM}_{0,(0)}(\BG_{i})}\frac{1}{{\bf w}-\psi} = \frac{\bf w}{|G_{i}|},
        \quad 
        \int_{\overline{\cM}_{0,(\mu,-\mu)}(\BG_{i})}\frac{\psi_2^a}{{\bf w}-\psi_1} =\frac{ (-{\bf w})^a}{|G_i|}&, \quad a\in \bZ_{\geq 0},
        \\
        \int_{\overline{\cM}_{0,(\mu,-\mu)}(\BG_{i})}\frac{1}{({\bf w}_1-\psi_1)({\bf w}_2-\psi_2)}= \frac{1}{({\bf w}_1+{\bf w}_2)|G_{i}|}&,
    \end{aligned}
    \end{equation*}
    for $i=-m,r$.
    
\end{prop}
\subsection{Open/descendant correspondence}
Given $u:(\Si,\partial\Si, {\bf x})\rightarrow (\cF(m,r),S^{1})$ that represents a point $\xi\in \overline{\cM}_{(g,h),n}(\cF(m,r),S^{1}|\beta',\vec{\mu})^T$,
we have 
\[
    \Si = C \cup \bigcup_{j=1}^h D_j,
\]
where $C$ is a closed nodal curve of genus $g$ with marked points $x_1,\dots,x_n$, and $D_j$'s are holomorphic disks.
$C$ and $D_j$ intersect at $y_j$. Let $\hat{u}= u|_C$ and $u_j=u|_{D_j}$.
Then $\hat{u}:(C,{\bf x}, {\bf y})\rightarrow \cF(m,r)$ represents a point $\hat{\xi}\in\overline{\cM}_{g,n+h}(\cF(m,r),\beta)^{T}$, 
where the $(n+j)$-th marked point $x_{n+j}$ is $y_j$.

The following definition is introduced in \cite[Definition 111]{Liu13}.

\begin{defn}[Decorated graphs]
    Define $G_{g,\vec{j}}(\cF(m,r),\beta)$ $(\vec{j}\in (\mathrm{Box}(\Sigma))^{n+h})$ to be the set of all decorated graphs $\hat{\Ga}=(\Ga,\vec{f},\vec{d},\vec{g},\vec{s},\vec{k})$ defined as follows. 
    Let $n\in\bZ_{\geq 0}$, $h\in \bZ_{>0}$ and $\beta=d[\cF(m,r)]\in E(\cF(m,r))$. A genus $g$, $(n+h)$-pointed,
    degree $\beta$ decorated graph for $\cF(m,r)$ is a tuple $\hat{\Ga}=(\Ga,\vec{f},\vec{d}, \vec{g},\vec{s},\vec{k})$ consisting of 
    the following data.
    \begin{enumerate}
        \item $\Gamma$ is a compact, connected 1-dimensional CW complex. Let $V(\Ga)$ denote
        the set of vertices in $\Ga$, denoted by $\bullet$. 
        Let $E(\Ga)$ denote the set of edges, where an edge $e$ is a line connecting two $\bullet$ vertices.
        Let $F(\Ga)$ be the set of flags:
        \[
            \{(e,v)\in E(\Ga)\times V(\Ga):v\in e\}.
        \]
        For each $v\in V(\Ga)$, let $F_v$ (\emph{resp.} $E_v$) denote the flags (\emph{resp.} edges) attached to $v$, and let $\val(v)=|F_v|=|E_v|$ denote the number of flags (\emph{resp.} edges) incident to $v$.
        \item The \emph{label map} $\vec{f}: V(\Ga)\rightarrow \{-m,r\}$ labels each $\bullet$ with a number.
            If $v_1,v_2\in V(\Ga)$ are connected by an edge, we require
            $\vec{f}(v_1)\neq \vec{f}(v_2)$.
        \item The \emph{degree map} $\vec{d}:E(\Ga)\rightarrow \bZ_{>0}$ sends an edge $e$ to a positive integer $\vec{d}(e)=d_e$.
        \item The \emph{genus map} $\vec{g}:V(\Ga)\rightarrow\bZ_{\geq 0}$ sends a vertex $v\in V(\Ga)$ to a nonnegative integer $g_v$.
        \item The \emph{marking map} $\vec{s}:\{1,2,\dots,n+h\}\rightarrow V(\Ga)$. For each $v\in V(\Ga)$, define $S_v := \vec{s}^{-1}(v)$, and $n_v=|S_v|$.
        \item $\vec{k}$ is the twisting map that sends each flag $(e,v)\in F(\Gamma)$ to some $k_{(e,v)}\in G_{\vec{f}(v)}$, and each marking $i\in\{1,\dots, n+h\}$ to some $k_{i}\in G_{\vec{f}(\vec{s}(i))}$.
    \end{enumerate}
    The data is required to satisfy the following conditions:
    \begin{itemize}
        \item [(i)] (genus) $g = \sum_{v\in V(\Ga)}g_v + |E(\Ga)| - |V(\Ga)| + 1$.
        \item [(ii)] (degree) $d = \sum_{e\in E(\Ga)}d_e$.
        \item [(iii)] (compatibility along an edge) For any edge $e\in E(\Ga)$, if $v, v'\in V(\Ga)$ are the two incident vertices, then $k_{(e,v)} \in G_{\vec{f}(v)}$ and $k_{(e,v')}\in G_{\vec{f}(v')}$ are determined by $\vec{d}(e)\in\bZ_{>0}$.
        \item[(iv)] (compatibility at a vertex) Given $v\in V(\Ga)$, let $E_v$ and $S_v$ be the set of edges and marked points, then
        \[
            \prod_{e\in E_v}k^{-1}_{(e,v)}\prod_{i\in S_v}k_i = 1.
        \]
        \item [(v)](compatibility with $\vec{j}$) For each $i=1,\dots, n+h$, the pair $(p_{\vec{f}(\vec{s}(i))}
,k_{i})$ represents a point in
the inertia component $\cF_{j_{i}}$.
    \end{itemize}
\end{defn}

Let $F_{\hat{\Ga}}$ be the connected component corresponding to $\hat{\Ga}$. We have
\[
    \overline{\cM}_{g,n+h}(\cF(m,r),\beta)^{T} = \bigcup_{\hat{\Ga}\in G_{g,n+h}(\cF(m,r),\beta)}F_{\hat{\Ga}}.
\]
By \cite[Theorem 137]{Liu13}, we get
\begin{prop}\label{prop:localization-liu} Let $\beta=d[\cF(m,r)]\in E(\cF(m,r))$. Then for $\gamma_1,\dots,\gamma_{n+h}\in H^*_{\CR,T}(\cF(m,r))$ and $g, a_1,\dots, a_{n+h}\in\bZ_{\geq 0}$, we have
    \begin{equation*}
    \begin{aligned}
        &\langle \tau_{a_1}(\gamma_1),\dots,\tau_{a_{n+h}}(\gamma_{n+h})\rangle^{\cF(m,r),T}_{g,n+h,\beta} 
            \\
            = &\sum_{\vec{j}\in{\rm Box}(\Sigma)^{n+h}}\sum_{\hat{\Ga}\in G_{g,\vec{j}}(\cF(m,r),\beta)}\frac{1}{|\Aut(\hat{\Ga})|}\prod_{e\in E(\Ga)}\frac{{\bf h}(e,d_e)}{d_e}
            \prod_{v\in V(\Ga)}\Big(\prod_{(e,v)\in E_v} {\bf h}(e,v)\prod_{i\in S_v}i^*_{p_{\vec{f}(v)}}\gamma_{i}\Big)
            \\
            &\cdot\prod_{v\in V(\Ga)}\int_{\overline{\cM}_{g_v,\vec{k}_{v}}(\BG_{\vec{f}(v)})}\frac{{\bf h}(v)\prod_{i\in S_v}\psi_i^{a_i}}{\prod_{e\in E_v}({\bf w}_{(e,v)}-\psi_{(e,v)})}.
    \end{aligned}
\end{equation*}
\end{prop}

Let $\widetilde{\Ga}$ and $\hat{\Ga}$ be the decorated graphs corresponding to the map $u$ and $\hat{u}$ respectively.
There is a map 
\begin{equation}\label{eq:graph-map}
    G_{g,\vec{j}}(\cF(m,r),S^{1}|\beta',\vec{\mu})\rightarrow G_{g,(\vec{j},\iota(l_1),\dots,\iota(l_h))}(\cF(m,r),\beta): \widetilde{\Ga}\mapsto\hat{\Ga},
\end{equation}    
by cutting all leaves $l_j$ of $\widetilde{\Ga}$, and then labeling the attached $\bullet$ vertex of $l_j$ with $n+j$, with the corresponding twisting map $k_{n+j}$ with $k_{(l_j,v)}^{-1}$.
Because the leaves $l_j$ are labeled, we have $\Aut(\widetilde{\Ga})=\Aut(\hat{\Ga})$.

By Proposition \ref{prop:localization}, Equation \eqref{eq:graph-map} and Proposition \ref{prop:localization-liu}, we give an open/descendant correspondence theorem:

\begin{theorem}\label{thm:open-descendant} Let $\beta' = d[\cF(m,r)] - \sum_{j\in J_-} \mu_j[D_1] + \sum_{j\in J_+}\mu_j[D_2]\in H_2(\cF(m,r),L)$. Then for $\gamma_1,\dots,\gamma_n\in H^*_{\CR,T}(\cF(m,r))$ and $g, a_1,\dots, a_n\in\bZ_{\geq 0}$, we have
    \begin{equation*}
    \begin{aligned}
   & \langle \tau_{a_1}(\gamma_1)\dots\tau_{a_n}(\gamma_n)\rangle^{\OGW}_{g,\beta',\vec{\mu}}
    = \prod_{j\in J_-}D^1(-\mu_j)\prod_{j\in J_+}D^2(\mu_j)\ \ \ \ \ \ \ \ \ \ \ \ \
    \\
    & \cdot \int_{[\overline{\cM}_{g,n+h}(\cF(m,r),\beta)]^{\vir,T}}\frac{\prod_{i=1}^{n}\psi_i^{a_i}\ev_i^*(\gamma_i)\prod_{j\in J_-}\ev_{n+j}^* {\one_{-\mu_j,-m}}\prod_{j\in J_+}\ev_{n+j}^* \one_{-\mu_j,r}}
        {\prod_{j=1}^h\frac{\bp}{\mu_j}(\frac{\bp}{\mu_j}-\psi_{n+j})},
    \end{aligned}
\end{equation*}
where 
\begin{equation*}\begin{aligned}
    &\one_{-\mu,r} := \one_{h,r}, \ & -\mu \equiv h\pmod r, \quad\quad 0\leq h\leq r-1.
\\
&\one_{-\mu,-m} := \one_{h,-m}, \ & -\mu \equiv h\pmod m, \ -m+1\leq h\leq 0.
\end{aligned}\end{equation*}
\end{theorem}

\subsection{Generating function of open Gromov-Witten invariants}\label{sec:gen.fun.}
Consider the generating function of genus $g$, $n$ boundary circles open Gromov-Witten invariants of $(\cF(m,r),L)$: let $\bt=tH + \sum_{i=-m+1}^{r-1}t^i\one_i$,
\[
    F_{g,n}(\bt,Q;\tX_1,\dots,\tX_n) = \sum_{\beta\in E(\cF(m,r))\atop {l\geq 0}}\sum_{\vec{\mu}=(\mu_1,\dots,\mu_n)\atop {\mu_i\in\bZ_{\neq 0}}}
    \frac{Q^\beta}{l!}\langle\bt^l\rangle^{\OGW}_{g,\beta',\vec{\mu}}\prod_{i=1}^{n}\tX_{i}^{\mu_i}.
\]
Suppose $h\in \{0,1,\dots,r-1\}, k\in \bZ$, we define
\[
    \Phi^{h,r}_k(\tX) = \sum_{\mu>0\atop {\mu \equiv h (\mod r)}}(\frac{\bp}{\mu})^{-(k+2)}D^2(\mu)\tX^\mu.
\]
Suppose $h\in \{-m+1,\dots, -1, 0\}, k\in\bZ$, we define
\[
    \Phi^{h,-m}_k(\tX) = \sum_{\mu<0\atop {\mu \equiv h (\mod m)}}(\frac{\bp}{\mu})^{-(k+2)}D^1(-\mu)\tX^\mu.
\]
Let $$\dsum_{\alpha\in \{0,\dots,m+r-1\}}\txi^{\alpha}_{k}(\tX)\phi_{\alpha} = \dsum_{h=0}^{r-1}\Phi^{h,r}_{k}(\tX)\one_{h,r}+\dsum_{h=0}^{m-1}\Phi^{-h,-m}_{k}(\tX)\one_{-h,-m}.$$
and $\txi^{\alpha}(\tX,z)=\dsum_{k\geq -2}\txi^{\alpha}_{k}(\tX)z^{k}$.
With the above notations and the Theorem \ref{thm:open-descendant}, we have
\begin{prop}\label{prop:open-generating-function}For $g\geq 0, n>0$,
\begin{equation*}
    \begin{aligned}
                & \ \ \ \ \ F_{g,n}(\bt,Q;\tX_{1},\dots,\tX_{n}) 
        \\
                &= \sum_{k_1,\dots,k_n\geq 0}\sum_{\alpha_{i}\in \{0,\dots,m+r-1\}}
        \llangle\phi_{\alpha_1}\psi^{k_1},\dots,\phi_{\alpha_n}\psi^{k_n}\rrangle_{g,n}^{\cF(m,r),T}
        \prod_{j=1}^{n}\tilde{\xi}_{k_j}^{\alpha_j}(\tX_j)
        \\
                &= [z_1^{-1}\dots z_n^{-1}]\sum_{\alpha_1,\dots,\alpha_n\in\{0,\dots,m+r-1\}}
        \llangle\frac{\phi_{\alpha_1}}{z_1-\psi_1},\dots,\frac{\phi_{\alpha_n}}{z_n-\psi_n}\rrangle_{g,n}^{\cF(m,r),T}
        \prod_{j=1}^{n}\tilde{\xi}^{\alpha_j}(z_j,\tX_j).
    \end{aligned}
\end{equation*}
\end{prop}

Let $\vec{\Ga}\in{\bf\Ga}_{g,n}(\cF(m,r))$ be a labeled graph defined in Section \ref{sec:Givental-graph}. To each ordinary leaf $l_j\in L^o(\Ga)$ with
$\beta(l_j)=\beta\in\{0,\dots,m+r-1\}$ and $k(l_j)=k\in \bZ_{\geq 0}$, we assign the following weight (open leaf)
\begin{equation}\label{eqn:A-model-open-leaf}
    (\tilde{\cL}^O)^\beta_k(l_j) = [z^k](\sum_{\alpha,\gamma\in\{0,\dots,m+r-1\}}
    \Big(\tilde{\xi}^\alpha(z,\tX_j)S^{\hat{\underline{\gamma}}}_{\ \alpha}(z)\Big)_{+}
    R(-z)_\gamma^{\ \beta})
\end{equation}
where $\big(\cdot\big)_{+}$ is the non-negative power part as a $z$-polynomial.

We define a new weight $\widetilde{\omega}^{O}_A(\vec{\Ga})$ of a labeled graph $\vec{\Ga}\in {\bf\Ga}_{g,n}(\cF(m,r))$ as 
\begin{equation*}
    \begin{aligned}
        \widetilde{\omega}^{O}_A(\vec{\Ga}) = &\prod_{v\in V(\Ga)}\Big(\sqrt{\Delta^{\beta(v)}(q)}\Big)^{2g(v)-2+\val(v)}\langle\prod_{h\in H(v)}\tau_{k(h)}\rangle_{g(v)}
        \prod_{e\in E(\Ga)}\cE^{\beta(v_1(e)),\beta(v_2(e))}_{k(h_1(e)),k(h_2(e))}
        \\
        & \cdot \prod_{l\in L^1(\Ga)}(\cL^1)^{\beta(l)}_{k(l)}(l)\prod_{j=1}^n(\tilde{\cL}^{O})^{\beta(l_j)}_{k(l_j)}(l_j).
    \end{aligned}
\end{equation*}
In \cite[Section 3.1]{Lan25} the descendant graph sum formula for $\cF(m,r)$ is established. By Proposition \ref{prop:open-generating-function}, we have the following open graph sum formula for $F_{g,n}$.
\begin{theorem}\label{thm:open-A-graph-sum}  \rm
    For $2g-2+n>0$, we have
    \[
        F_{g,n}(\bt,Q;\tX_{1},\dots,\tX_{n}) = \sum_{\vec{\Ga}\in {\bf \Ga}_{g,n}(\cF(m,r))}\frac{\widetilde{\omega}_A^O(\vec{\Ga})}{|\Aut(\vec{\Ga})|}.
    \]
\end{theorem}

Theorem \ref{thm:open-A-graph-sum} implies
\begin{coro}\label{thm:open-A-graph-sum-Qone}\rm
    Let $\omega_A^O(\vec{\Ga}) = \widetilde{\omega}_A^O(\vec{\Ga})\big|_{Q=1}$. For $2g-2+n>0$, we have
    \[
        F_{g,n}(\bt,1;\tX_{1},\dots,\tX_{n}) = \sum_{\vec{\Ga}\in {\bf \Ga}_{g,n}(\cF(m,r))}\frac{\omega_A^O(\vec{\Ga})}{|\Aut(\vec{\Ga})|}.
    \]
\end{coro}
\begin{convention}
    From this point onwards, we use the specialization $\bt=\tau(\bq)$ in the calculation, where $\tau$ is the equivariant mirror map.
\end{convention}

\subsection{Explicit formula of $F_{0,1}$}\label{sec:A-explicit-disk}
Let $\one^*_{h,r}$ (resp. $\one^*_{h,-m}$) be the dual basis of $\one_{h,r}$ (resp. $\one_{h,-m}$) under the orbifold Poincar\'{e} pairing $(,)_{\cF(m,r),T}$:
$$
  \one_{h,r} = \frac{\one^{*}_{r\langle\frac{-h}{r}\rangle,r}}{r^{1-\delta_{0,h}}\bp^{\delta_{0,h}}},\quad
  \one_{h,-m} = \frac{\one^{*}_{-m\langle\frac{h}{m}\rangle,-m}}{m^{1-\delta_{0,h}}(-\bp)^{\delta_{0,h}}}.
$$

We write
\[
    F_{0,1}(\bt,Q;\tX) = \sum_{-m+1\leq h\leq 0} F_{0,1}^{-,h} + \sum_{0\leq h\leq r-1}F_{0,1}^{+,h}.
\]
\begin{align*}
    F_{0,1}^{+,h}(\bt,Q;\tX) &= \sum_{a\geq 0\atop {\mu = ra+h >0}} \frac{1}{r^{1-\delta_{0,h}}\bp^{\delta_{0,h}}}\llangle\frac{\one^*_{h,r}} { \frac{\bp}{\mu} (\frac{\bp}{\mu}-\psi)}\rrangle_{0,1}^{\cF(m,r),T} D^2(\mu) X^\mu
    \\
    &= \sum_{a\geq 0\atop {\mu = ra+h >0}}\sum_{k\geq 0} \frac{1}{r^{1-\delta_{0,h}}\bp^{\delta_{0,h}}}\llangle\one^*_{h,r} \psi^k\rrangle_{0,1}^{\cF(m,r),T}\big(\frac{\bp}{\mu}\big)^{-(k+2)}D^2(\mu) \tX^\mu
    \\
    &= \ \sum_{k\geq 0}\frac{1}{r^{1-\delta_{0,h}}\bp^{\delta_{0,h}}}\llangle\one^*_{h,r} \psi^k\rrangle_{0,1}^{\cF(m,r),T}\Phi_k^{h,r}(\tX).
\end{align*}

\begin{align*}
    F_{0,1}^{-,h}(\bt,Q;\tX) &= \sum_{a\geq 0\atop {\mu = -ma+h <0}} \frac{1}{m^{1-\delta_{0,h}}(-\bp)^{\delta_{0,h}}}\llangle\frac{\one^*_{h,-m}} { \frac{\bp}{\mu} (\frac{\bp}{\mu}-\psi)}\rrangle_{0,1}^{\cF(m,r),T} D^1(-\mu) \tX^\mu
    \\
    &= \sum_{a\geq 0\atop {\mu = -ma+h <0}}\sum_{k\geq 0}\frac{1}{m^{1-\delta_{0,h}}(-\bp)^{\delta_{0,h}}}\llangle\one^*_{h,-m} \psi^k\rrangle_{0,1}^{\cF(m,r),T}\big(\frac{\bp}{\mu}\big)^{-(k+2)}D^1(-\mu) \tX^\mu
    \\
    &= \ \sum_{k\geq 0}\frac{1}{m^{1-\delta_{0,h}}(-\bp)^{\delta_{0,h}}}\llangle\one^*_{h,-m} \psi^k\rrangle_{0,1}^{\cF(m,r),T}\Phi_k^{h,-m}(\tX).
\end{align*}

\subsubsection{Winding number is positive}
\[
 F_{0,1}^{+,h}(\bt,1;\tX) = \sum_{\mu=ra+h>0\atop{a\geq 0}} \frac{I_{+,h}(\bq,\frac{\bp}{\mu})}{r^{1-\delta_{0,h}}\bp^{\delta_{0,h}}}D^2(\mu)\tX^\mu.
\]
By \eqref{eqn:I-function}, we have
\begin{equation*}
    I_{+,h}(\bq,\frac{\bp}{\mu})= e^{\frac{-\sum_{a\in I_0}\mu{\bf w}_a\log q_a}{\bp}}\sum_{d\in \bK_{\eff,r}\atop{r\{-\tilde{d}_r\}=h}}e^{-\frac{\mu\bp_{-}\log q_*}{\bp}}\bq^d(\frac{\mu}{\bp})^{\lceil\deg d\rceil}\cdot \frac{\Ga(\frac{\mu}{r}+1-\{-\tilde{d}_r\})}{\tilde{d}_{-m}!\prod_{i\in I_0}d_i!\Ga(\frac{\mu}{r}+\tilde{d}_r+1)}.
\end{equation*}

For $1\leq h\leq r-1$, let $\mu = rk+h$, $k\geq 0$, and $-\tilde{d}_r= -b + \frac{h}{r}$,
\begin{equation*}
  \begin{aligned}
    F_{0,1}^{+,h}(\bt,1;\tX) = \frac{1}{\bp}\sum_{k\geq 0, \mu=rk+h\atop{d\in \bK_{\eff,r}\atop {-\tilde{d}_r=-b+\frac{h}{r}}}}e^{\frac{-\sum_{a\in I_0}\mu{\bf w}_a\log q_a}{\bp}}e^{-\frac{\mu\bp_{-}\log q_*}{\bp}}\bq^d \frac{(\frac{\mu}{\bp})^{k-2+\lceil\deg d\rceil}}{\tilde{d}_{-m}!\prod_{i\in I_0}d_i!(k+b)!}\tX^\mu
  \end{aligned}
\end{equation*}
For $h=0$, let $\mu = rk$, $k\geq 1$.
\begin{equation*}
  \begin{aligned}
    F_{0,1}^{+,0}(\bt,1;\tX) = \frac{1}{\bp}\sum_{k\geq 1, \mu=rk\atop{d\in \bK_{\eff,r}\atop{\tilde{d}_r\in\bZ}}}e^{\frac{-\sum_{a\in I_0}\mu{\bf w}_a\log q_a}{\bp}}e^{-\frac{\mu\bp_{-}\log q_*}{\bp}}\bq^d\frac{(\frac{\mu}{\bp})^{k-2+\lceil\deg d\rceil}}{\tilde{d}_{-m}!\prod_{i\in I_0}d_i!(k+\tilde{d}_{r})!}\tX^\mu
  \end{aligned}
\end{equation*}
In summary, for $0\leq h\leq r-1$, we have
\[
  F_{0,1}^{+,h}(\bt,1;\tX) = \frac{1}{\bp}\sum_{k\geq 0, \mu=rk+h>0\atop{d\in \bK_{\eff,r}\atop{r\{-\tilde{d}_r\}=h}}}e^{\frac{-\sum_{a\in I_0}\mu{\bf w}_a\log q_a}{\bp}}e^{-\frac{\mu\bp_{-}\log q_{*}}{\bp}}\bq^d\frac{(\frac{\mu}{\bp})^{k-2+\lceil\deg d\rceil}}{\tilde{d}_{-m}!\prod_{i\in I_0}d_i!(k+\lceil\tilde{d}_r\rceil)!}\tX^\mu.
\]

\subsubsection{Winding number is negative}
Similarly to the positive part, we have
\[
 F_{0,1}^{-,h}(\bt,1;\tX) = \sum_{\mu=-mk+h<0\atop{k\geq 0}} \frac{I_{-,h}(\bq,\frac{\bp}{\mu})}{m^{1-\delta_{0,h}}(-\bp)^{\delta_{0,h}}}D^1(-\mu)\tX^\mu,
\]
where
\[
  I_{-,h}(\bq,\frac{\bp}{\mu}) = e^{\frac{-\sum_{a\in I_0}\mu{\bf w}_a\log q_a}{\bp}}\sum_{d\in \bK_{\eff,-m}\atop{\{-\tilde{d}_{-m}\}=-\frac{h}{m}}}e^{-\frac{\mu\bp_{+}\log q_*}{\bp}}\bq^d(\frac{\mu}{\bp})^{\lceil\deg d\rceil}\cdot \frac{\Ga(\frac{-\mu}{m}+1-\{-\tilde{d}_{-m}\})}{\tilde{d}_{r}!\prod_{i\in I_0}d_i!\Ga(\frac{-\mu}{m}+\tilde{d}_{-m}+1)}.
\]
Therefore,
\[
  F_{0,1}^{-,h}(\bt,1;\tX) = \frac{1}{\bp}\sum_{k\geq 0, \mu=-mk+h<0\atop{d\in \bK_{\eff,-m}\atop{\{-\tilde{d}_{-m}\}=-\frac{h}{m}}}}e^{\frac{-\sum_{a\in I_0}\mu{\bf w}_a\log q_a}{\bp}}e^{-\frac{\mu\bp_{+}\log q_{*}}{\bp}}\bq^d\frac{(\frac{\mu}{\bp})^{k-2+\lceil\deg d\rceil}}{\tilde{d}_{r}!\prod_{i\in I_0}d_i!(k+\lceil\tilde{d}_{-m}\rceil)!}\tX^\mu.
\]

\section{Mirror curve and topological recursion}\label{sec:Bmodel}
\subsection{The equivariant superpotential and the Frobenius structure of the Jacobian ring}\label{sec:W_T}
Let $Y$ be a coordinate on $\bC^*$.
The $T$-equivariant superpotential $W_T:\bC^*\rightarrow \bC$ is defined as
\[
  W_T(Y) = Y^r + \sum_{l=1}^{r-1} \tilde{q}_lY^l + \tilde{q}_0 + \sum_{l=1}^{m} \tilde{q}_{-l}Y^{-l} - \bw_r\log(Y^r) - \sum_{l=1}^{r-1}\bw_l\log(\tilde{q}_lY^l) - \sum_{l=0}^m\bw_{-l}\log(\tilde{q}_{-l}Y^{-l}).
\] 
where $\tilde{q}_{i}$
are complex parameters. Let $Y=e^y$, the Jacobian ring of $W_T$ is 
\[
    \Jac(W_T) := \bC[Y,Y^{-1},\bp,q]\Big/{\Big\langle\frac{\partial W_T}{\partial y}\Big\rangle}.
\]
Let $\{P_\alpha\}$ be the set of critical points of $W_T$.
Endow a metric on $\Jac(W_T)$ by the residue pairing 
\[
    (f,g) = \sum_{P_\alpha}\Res_{Y=P_\alpha} \frac{f(Y)g(Y)}{\partial W_T/\partial y}\frac{dY}{Y}.
\]
We have the following isomorphism as Frobenius algebras:
\begin{prop}[\cite{CCK,CCIT15}]\label{prop:Frobenius}
    \begin{equation*}
        QH_{T}^{*}(\cF(m,r),\bC)\cong \mathrm{Jac}(W_{T}(Y))
    \end{equation*}
    as Frobenius algebras.
\end{prop}

\subsection{Mirror curve}\label{sec:curve}
Consider the spectral curve $C_q$ defined as follows:
\[
    C_q = \{(x,y)\in\bC^2: x = W_T(e^y)\}.
\]
Let $\{P_\alpha\}_{\alpha=0}^{m+r-1}$ denote the set of critical points $x$ (i.e. $dx(P_\alpha)=0$) and let $B(Y_1,Y_2)$ be a meromorphic bidifferential on $C_q$ defined by
\[
    B(Y_1,Y_2) = \frac{dY_1dY_2}{(Y_1-Y_2)^2}.
\]
Let $\omega_{g,n}$ be the differential forms
defined recursively by the Chekhov-Eynard-Orantin topological recursion \cite{EO07}:
$$\omega_{0,1}=0,\quad \omega_{0,2}=B(Y_1,Y_2).$$
For $2g-2+n>0$,
\begin{equation*}
    \begin{aligned}
        \omega_{g,n}(Y_1,\dots,Y_n) = \sum_{\alpha=0}^{m+r-1}
        &\Res_{Y\rightarrow P_\alpha}\frac{\int_{\xi=Y}^{\hat{Y}}B(Y_n,\xi)}{2(\log(Y)-\log(\hat{Y}))dx}\Big(
            \omega_{g-1,n+1}(Y,\hat{Y},Y_1,\dots,Y_{n-1})
        \\
        &+ 
            \sum_{g_1+g_2=g}\sum_{I\sqcup J = \{1,\dots,n-1\}}\omega_{g_1,|I|+1}(Y,Y_{I})\omega_{g_2,|J|+1}(\hat{Y},Y_J)
        \Big),
    \end{aligned}
\end{equation*}
where $Y\neq P_\alpha$ is in a neighbourhood of $P_\alpha$, and $\hat{Y}\neq Y$ is the conjugate point of $Y$ such that $x(\hat{Y})=x(Y)$.

Near any critical point $v^\alpha(=\log{P_{\alpha}})$, we choose local coordinates $\zeta_\alpha$ and parameters $h_k^\alpha$ such that
$$x=u^\alpha+\zeta_\alpha^2, \quad y=v^\alpha+\sum_{k=1}^\infty h_k^\alpha \zeta_\alpha^k.$$
Let $\lambda = xdy$ be the Liouville form on $\bC^2$, $\Phi:=\lambda|_{C_q}$. Around $(P_\alpha, P_\beta)$, we expand $B(Y_1,Y_2)$ as
\[B(\zeta_\alpha,\zeta_\beta)=\Big( \frac{\delta_{\alpha\beta}}{(\zeta_\alpha -\zeta_\beta)^2}+ \sum_{k,\ell\geq 0} B_{k,\ell}^{\alpha,\beta}\zeta_\alpha^k\zeta_\beta^\ell\Big) d\zeta_\alpha\otimes d\zeta_\beta.\]

We introduce the following notations:
\begin{enumerate}
    \item Given $\alpha \in\{0,\dots,m+r-1\}$ and $k\in \bZ_{\geq 0}$, define
        \[
            d\xi^\alpha_k=-(2k-1)!!2^{-k}\Res_{P'\mapsto P_\alpha}B(P,P')\zeta_\alpha^{-2k-1},
        \]
        and let $\xi_{\alpha,0}$ be a function defined by $$\xi_{\alpha,0}=\sqrt{\frac{2}{\Delta^\alpha(\bq)}}\frac{P_\alpha}{Y-P_\alpha}.$$
        Note that $d\xi_{\alpha,0}=d\xi^\alpha_0$.
    \item Given $\alpha \in\{0,\dots,m+r-1\}$ and $k\in \bZ_{\geq 0}$, define    
        \[
            W^\alpha_k := d((-\frac{d}{dx})^k(\xi_{\alpha,0})),
        \]
        \[
            \theta_\alpha(z) := \sum_{k=0}^{\infty}d\xi^\alpha_kz^k,\quad
            \hat\theta_\alpha(z):= \sum_{k=0}^\infty W^\alpha_kz^k.
        \]
    \item The B-model $R$-matrix $\check{R}_\beta^{\ \alpha}(z)$ (which is a power series of $z$) is defined by asymptotic expansion
        \[
            \check{R}_\beta^{\ \alpha}(z) \sim \frac{\sqrt{-z}e^{-\check{u}^\alpha/z}}{2\sqrt{\pi}}\int_{\gamma_\alpha}e^{x/z}d\xi^{\beta}_{0}.
        \]
        Here $\gamma_\alpha$ is the Lefschetz thimble of the map $x$, i.e. $x(\gamma_\alpha)-\check{u}^\alpha \in \bR_{\geq 0}$.
    \item Given $\alpha,\beta\in\{0,\dots,m+r-1\}$, $k,l\geq 0$, define
        \begin{equation*}
            \begin{aligned}
                \check{B}^{\alpha,\beta}_{k,l} &:= \frac{(2k-1)!!(2l-1)!!}{2^{k+l+1}}B^{\alpha,\beta}_{2k,2l}
                \\
                &\ = 
                [z^kw^l]\Big(\frac{1}{z+w}\Big(\delta_{\alpha\beta}-\sum_{\gamma\in\{0,\dots,m+r-1\}}\check{R}_\gamma^{\ \alpha}(-z)\check{R}_\gamma^{\ \beta}(-w)\Big)\Big).
            \end{aligned}
        \end{equation*}
    \item Given $\alpha\in\{0,\dots,m+r-1\}$ and $k\in \bZ_{\geq 1}$, define     
        \begin{equation*}
            \begin{aligned}
                \check{h}^\alpha_k &:= -\frac{(2k-1)!!}{2^{k-1}}h^\alpha_{2k-1}
                \\
                & \ = [z^{k-1}]\Big(\sum_\beta h^\beta_1\check{R}_\beta^{\ \alpha}(-z)\Big).
            \end{aligned}
        \end{equation*}
\end{enumerate}

The B-model invariants $\omega_{g,n}$ can be expressed as graph sums. Given a labeled graph $\vec{\Gamma}\in \mathbf{ \Gamma}_{g,n}(\cF(m,r))$ with $L^o(\Gamma)=\{\ell_1,\cdots,\ell_n\}$. 
Define the vertex factors to be $$\check{\mathcal{V}}^{\beta(v)}_{g(v)}(v)=\Big(\dfrac{h_1^{\beta(v)}}{\sqrt{2}} \Big)^{2-2g(v)-\mathrm{val}(v)} \big<\dprod_{h\in H(v)} \tau_{k(h)}\big>_{g(v)}.$$
We define its weight to be
\[
\begin{split}
\check{w}_B(\vec{\Gamma}) = &\  \dprod_{v\in V (\Gamma)} \check{\mathcal{V}}^{\beta(v)}_{g(v)}(v)
\dprod_{e\in E(\Gamma)} \check{B}_{k(e),\ell(e)}^{\beta(v_1(e)),\beta(v_2(e))} (e) \\ & \cdot \dprod_{j=1}^n \dfrac{-1}{\sqrt{2}}d\xi_{k(\ell_j)}^{\beta(\ell_j)}(Y_j) \dprod_{\ell\in L^1(\Gamma)} \dfrac{-1}{\sqrt{2}}\check{h}_{k(\ell)}^{\beta(\ell)}.
\end{split}
\]

We cite here the Theorem 3.7 in \cite{DOSS}.
\begin{theorem}\label{thm:DOSS}
For $2g-2+n>0$, it holds that
\[
\omega_{g,n}=\dsum_{\vec{\Gamma}\in \mathbf{\Gamma}_{g,n}(\cF(m,r))} \dfrac{\check{w}_B(\vec{\Gamma})}{|\mathrm{Aut}(\vec{\Gamma})|}.
\]
\end{theorem}

\subsection{B-model open potentials}\label{sec} The function $X=e^{-x/\bp}$ has an essential singularity at $Y=0$:
\begin{equation*}
  \begin{aligned}
    X=e^{-x/\bp} =  Y\prod_{l=-m}^{r-1}\tilde{q}_l^{\frac{\bw_l}{\bp}}\exp\left\{-\Big(Y^r + \sum_{l=1}^{r-1} \tilde{q}_lY^l + \tilde{q}_0 + \sum_{l=1}^{m} \tilde{q}_{-l}Y^{-l}\Big)\Big/\bp\right\}.
  \end{aligned}
\end{equation*}
Let $D_\epsilon$ be a punctured disk around $Y=0$: 
\[
    D_\epsilon := \{Y\in\bC^*: 0<|Y|<\ep\},
\]
where $\ep$ is a small positive real number such that $D_\epsilon$ does not contain
any ramification points $\{P_\alpha\}_{\alpha=0}^{m+r-1}$.
In the punctured disk $D_\epsilon$, $Y$ can be expanded as a power series of $X$ by the Lagrange inversion theorem.
At $Y=0$, we define the operator $\fh^\circ_X$ and $\fh^\bullet_X$ that expand holomorphic functions $f(Y)$ and holomorphic differential forms 
$\theta(Y)$ on $D_\epsilon$ into Laurent series in $X$, based on their local behavior at $Y=0$:
\begin{itemize}
    \item For a holomorphic function $f(Y)$ on $D_\epsilon$, define $\fh^\circ_X:\cO(D_\ep)\rightarrow \bC[\![X^{\pm 1}]\!]$ by
    \[
        \fh^\circ_X(f) := \sum_{\mu\in\bZ_{\neq 0}} \Big(\mathop{\Res}\limits_{Y\rightarrow 0}f(Y)X^{-\mu}\frac{dX}{X}\Big)X^{\mu}.
    \]
    \item For a holomorphic differential form $\theta(Y)$ on $D_\epsilon$, define $\fh^\bullet_X:\Omega^1(D_\ep)\rightarrow \bC[\![X^{\pm 1}]\!]$ by
    \[     
        \fh_X^\bullet(\theta) := \sum_{\mu\in\bZ_{\neq 0}} \Big(\mathop{\Res}\limits_{Y\rightarrow 0} \theta(Y) X^{-\mu}\Big)\frac{X^{\mu}}{\mu}.
    \]
    If $f$ is a single-valued holomorphic function on $D_\ep$ such that $\theta(Y)=df$, then $\fh_X^\bullet(\theta)=\fh_X^\circ(f)$ by the integrations by parts.
\end{itemize}
For multi-holomorphic functions (\emph{resp.} forms) on $(D_\epsilon)^{\times n}$, we define $\fh_{X_1,\dots,X_n}^{\circ}$ and $\fh_{X_1,\dots,X_n}^\bullet$ as follows:
\begin{itemize}
    \item Let $f(Y_1,\dots,Y_n)$ be a holomorphic function on $(D_\epsilon)^{\times n}$, we define
    \[
        \fh^\circ_{X_1,\dots,X_n}(f) := \sum_{\mu_1,\dots,\mu_n\in\bZ_{\neq 0}}
        \Big(
            \mathop{\Res}\limits_{Y_1\rightarrow 0}\cdots \mathop{\Res}\limits_{Y_n\rightarrow 0}f(Y_1,\dots,Y_n)\prod_{i=1}^{n}X^{-\mu_i}_i\frac{dX_i}{X_i}
        \Big)\prod_{i=1}^{n}X^{\mu_i}_i.
    \]
    \item Let $\theta(Y_1,\dots,Y_n)$
    be a holomorphic differential form from the set of sections $\Gamma((D_\ep)^{\times n}, \omega_{C_q}^{\boxtimes n})$, we define
    \[  
        \fh^\bullet_{X_1,\dots,X_n}(\theta) := \sum_{\mu_1,\dots,\mu_n\in\bZ_{\neq 0}}
        \Big(
            \mathop{\Res}\limits_{Y_1\rightarrow 0}\cdots \mathop{\Res}\limits_{Y_n\rightarrow 0}\theta(Y_1,\dots,Y_n)\prod_{i=1}^{n}X^{-\mu_i}_i
        \Big)\prod_{i=1}^{n}\frac{X^{\mu_i}_i}{\mu_i}.
    \] 
\end{itemize}

We define the B-model open potentials as follows,
\begin{itemize}
    \item Define the B-model disk potential $W_{0,1}(\bold{q},X)$ as a Laurent series in $X$ with constant term zero such that
        \[
            \Big(\frac{1}{\bp}X\frac{d}{d X}\Big)W_{0,1}(\bold{q},X) = \fh_X^\bullet\Big(\frac{dY}{Y}\Big). 
        \]
    \item Let $\tilde{\omega}_{0,2}(Y_1,Y_2)$ be the meromorphic 2-form on $\bC^*\times \bC^*$,
    $$\tilde{\omega}_{0,2}(Y_1,Y_2):=\omega_{0,2}(Y_1,Y_2)-\frac{dX_1dX_2}{(X_1-X_2)^2}.$$
    Define the B-model annulus invariants by
        \[
            W_{0,2}(\bold{q},X_1,X_2) = \fh^\bullet_{X_1,X_2}(\tilde{\omega}_{0,2}(Y_1,Y_2)),
        \]
    where $\omega_{0,2}(Y_1,Y_2)$ is expanded as a series in the region $|Y_1|>|Y_2|$, and $\frac{dX_1dX_2}{(X_1-X_2)^2}$ is expanded in the region $|X_1|>|X_2|$.
    \item For $2g-2+n>0$, $\omega_{g,n}$ is holomorphic on $(D_\ep)^{\times n}$, we define
        \[
            W_{g,n}(\bold{q}, X_1,\dots,X_n) = \fh^\bullet_{X_1,\dots,X_n}(\omega_{g,n}).
        \]
\end{itemize}
For a labeled graph $\vec{\Ga}\in{\bf\Ga}_{g,n}(\cF(m,r))$ with $L^o(\Ga) = \{l_1,\dots,l_n\}$, we introduce the B-model open leaves
as 
\[
    (\check{\cL}^O)^{\beta(l_j)}_{k(l_j)}(l_j) = \frac{-1}{\sqrt{2}}\fh^\bullet_{X_j}(d\xi_{k(l_j)}^{\beta(l_j)}(Y_j)).
\]  
We assign the B-model open weights of $\vec\Ga$ as
\[
\begin{split}
\check{w}_B^O(\vec{\Gamma}) = &\ \dprod_{v\in V (\Gamma)} \check{\mathcal{V}}^{\beta(v)}_{g(v)}(v)
\dprod_{e\in E(\Gamma)} \check{B}_{k(e),\ell(e)}^{\beta(v_1(e)),\beta(v_2(e))} (e) \\ & \cdot \dprod_{j=1}^n \frac{-1}{\sqrt{2}}\fh_{X_j}^\bullet(d\xi_{k(l_j)}^{\beta(l_j)}(Y_j)) \dprod_{\ell\in L^1(\Gamma)}\dfrac{-1}{\sqrt{2}} \check{h}_{k(\ell)}^{\beta(\ell)}.
\end{split}
\]
Following Theorem \ref{thm:DOSS}, we have
\begin{theorem} \label{thm:B-graph-sum}
  For $2g-2+n>0$, it holds that
  \[
      W_{g,n}(\bold{q},X_1,\dots,X_n) = \dsum_{\vec{\Gamma}\in \mathbf{\Gamma}_{g,n}(\cF(m,r))} \dfrac{\check{w}_B^{O}(\vec{\Gamma})}{|\mathrm{Aut}(\vec{\Gamma})|}.
  \]
\end{theorem}

\subsection{Explicit B-model disk potential}\label{sec:B-explicit-disk}
In this subsection, we compute the B-model disk potential explicitly. Let 
$$\Phi_0 = \frac{\tilde{q}_0}{\tilde{q}_0-{\bf w}_0}\frac{\partial\Phi}{\partial \tilde{q}_0}=\frac{dY}{Y}.$$
and let 
\[
  \fh^\bullet_X(\Phi_0) = \sum_{\mu\in\bZ_{\neq 0}} R_\mu X^\mu.
\]

Consider the residue
\begin{equation*}
  \begin{aligned}
    R_\mu &= \frac{1}{\mu}\Res_{Y=0} X^{-\mu}\frac{dY}{Y}
    \\
    &= [Y^{\mu}]\frac{e^{\mu\tilde{q}_0/\bp}}{\mu}\prod_{l=-m}^{r-1}\tilde{q}_l^{-\mu\frac{\bw_l}{\bp}}\cdot\exp\left\{\mu\Big(Y^r + \sum_{l=1}^{r-1} \tilde{q}_lY^l + \sum_{l=1}^{m} \tilde{q}_{-l}Y^{-l}\Big)/\bp\right\}
    \\
    &= \frac{e^{\mu\tilde{q}_0/\bp}}{\mu}\prod_{l=-m}^{r-1}\tilde{q}_l^{-\mu\frac{\bw_l}{\bp}}\cdot\sum_{a_l\in\bZ_{\geq 0}\atop{l=-m,\dots,-1, 1,\dots,r\atop{\sum_{l=1}^{r}la_l}-\sum_{l=1}^mla_{-l}=\mu}}\frac{(\frac{\mu}{\bp})^{\sum_{l=1}^r a_l +\sum_{l=1}^ma_{-l}}\prod_{l=1}^{r-1}\tilde{q}_l^{a_l}\prod_{l=1}^{m}\tilde{q}_{-l}^{a_{-l}}}{\prod_{l=1}^{r}a_l!\prod_{l=1}^{m}a_{-l}!}
    \\
    &= \frac{1}{\mu}\prod_{l=-m}^{r-1}\tilde{q}_l^{-\mu\frac{\bw_l}{\bp}}\cdot\sum_{a_l\in\bZ_{\geq 0}, l\in I \atop{{\sum_{l=1}^{r}la_l}-\sum_{l=1}^mla_{-l}=\mu}}\frac{(\frac{\mu}{\bp})^{\sum_{l\in I} a_l}\prod_{l=-m}^{r-1}\tilde{q}_l^{a_l}}{\prod_{l\in I}a_l!},
  \end{aligned}
\end{equation*}
where $I$ is the index set of integers from $-m$ to $r$.

Let 
\begin{equation*}
  \begin{aligned}
    &\tilde{q}_i = q_i, \quad 0 \leq i \leq r-1
    \\
    &\tilde{q}_i = q_*^{-i}q_i,\quad -m+1\leq i\leq -1 
    \\
    &\tilde{q}_{-m} = q_*^m,
  \end{aligned}
\end{equation*}
then we have
\begin{eqnarray*}
      \prod_{l=-m}^{r-1}\tilde{q}_l^{-\mu\frac{\bw_l}{\bp}} &=& q_*^{-\frac{\mu\bp_-}{\bp}}\prod_{l=-m+1}^{r-1}q_l^{-\mu\frac{\bw_l}{\bp}},
      \\
  \prod_{l=-m}^{r-1}\tilde{q}_l^{a_l} &=& q_*^{\sum_{i=1}^{m}ia_{-i}}\prod_{l=-m+1}^{r-1}q_l^{a_l}.
\end{eqnarray*}
Given a positive integer $\mu >0$, we assume
$\mu = rk+h$, where $0\leq h\leq r-1$ and $k\geq 0$, then
\[
  R_\mu = \sum_{a_l\in\bZ_{\geq 0}, l\in I\atop {{\sum_{l=1}^{r}la_l}-\sum_{l=1}^mla_{-l}=rk+h}}\frac{1}{\mu}q_*^{-\frac{\mu\bp_-}{\bp}+ \sum_{i=1}^{m}ia_{-i}}\prod_{l=-m+1}^{r-1}q_l^{-\mu\frac{\bw_l}{\bp}+a_l}
  \frac{(\frac{\mu}{\bp})^{\sum_{l\in I}a_l}}{\prod_{l\in I}a_l!}.
\]
By the bijection:
\begin{equation*}
  \begin{aligned}
      &a_{-m} \rightarrow \tilde{d}_{-m},
      \\
      &a_i \rightarrow d_i, \quad -m+1\leq i \leq r-1,
      \\
      &a_r \rightarrow k+\lceil \tilde{d}_r\rceil,
  \end{aligned}
\end{equation*}
we get the following formula
\begin{equation*}
  \begin{aligned}
    R_\mu = \sum_{d\in \bK_{\eff,r}\atop{\mu = rk+h, \atop{ r\{-\tilde{d}_r\} = h}}}\frac{1}{\mu}q_*^{-\frac{\mu\bp_-}{\bp}}(\prod_{l\in I_0}q_l^{-\mu\frac{\bw_l}{\bp}})
    \bq^d 
  \frac{(\frac{\mu}{\bp})^{k+\lceil\deg d\rceil}}{\tilde{d}_{-m}!\prod_{l\in I_0}d_l!(k+\lceil\tilde{d}_r\rceil)!}. 
  \end{aligned}
\end{equation*}
Similarly, given a negative integer $\mu<0$, we assume
$\mu = -mk+h$, where $-m+1\leq h\leq 0$ and $k\geq 0$, we have
\[
  R_\mu = \sum_{a_l\in\bZ_{\geq 0}, l\in I\atop {{\sum_{l=1}^{r}la_l}-\sum_{l=1}^mla_{-l}=-mk+h}}\frac{1}{\mu}q_*^{-\frac{\mu\bp_-}{\bp}+ \sum_{i=1}^{m}ia_{-i}}\prod_{l=-m+1}^{r-1}q_l^{-\mu\frac{\bw_l}{\bp}+a_l}
  \frac{(\frac{\mu}{\bp})^{\sum_{l\in I}a_l}}{\prod_{l\in I}a_l!} .
\]
By the bijection:
\begin{equation*}
  \begin{aligned}
      &a_{-m} \rightarrow k+\lceil \tilde{d}_{-m}\rceil,
      \\
      &a_i \rightarrow d_i, \quad -m+1\leq i \leq r-1,
      \\
      &a_r \rightarrow \tilde{d}_{r},
  \end{aligned}
\end{equation*}
we have
\begin{equation*}
  \begin{aligned}
    R_\mu = \sum_{d\in \bK_{\eff,-m}\atop{\mu = -mk+h, \atop{m\{-\tilde{d}_{-m}\} = -h}}}\frac{1}{\mu}q_*^{-\frac{\mu\bp_+}{\bp}}(\prod_{l\in I_0}q_l^{-\mu\frac{\bw_l}{\bp}})
    \bq^d 
  \frac{(\frac{\mu}{\bp})^{k+\lceil\deg d\rceil}}{\tilde{d}_{r}!\prod_{l\in I_0}d_l!(k+\lceil\tilde{d}_{-m}\rceil)!}. 
  \end{aligned}
\end{equation*}
This uses the fact:
\begin{equation*}
    \sum_{i=1}^{m-1}id_{-i}+mk+m\lceil\tilde{d}_{-m}\rceil=d_*-\mu.
\end{equation*}

\section{Mirror Symmetry}\label{sec:MS}
In this section, we establish the open mirror symmetry for $\cF(m,r)$. We use $\tX_{i}$ in A-model and $X_{i}$ in B-model. In equations relating A-model and B-model, these variables are identified, i.e. $\tX_{i}=X_{i}$.
\subsection{Identification of open leaves}

\begin{theorem}\label{thm:disk}
    $F_{0,1}(\tau(\bq),1,\tX)=W_{0,1}(\bold{q},X)$
\end{theorem}
\begin{proof}
    We use the fact:
    \begin{equation*}
        \fh^\bullet_{X}(\Phi_{0})=(\frac{1}{\bp}\tX\frac{d}{d\tX})F_{0,1}(\tau(\bq),1;\tX),
    \end{equation*}
    which follows from Section \ref{sec:A-explicit-disk} and Section \ref{sec:B-explicit-disk}.
\end{proof}
We define:
\[
    U^{\hat\beta}(z)(\tau(\bold{q}),\tX) := \sum_{\alpha} \tilde{\xi}^\alpha(z,\tX)(\hat\phi_{\beta}(\bq),S(\phi_\alpha)(z))\Big|_{Q=1}.
\]
For $m\in\ZZ_{\geq -2}$, 
\[
    \left(\frac{1}{\bp}\tX\frac{d}{d\tX}\right)[z^{m}](U^{\hat\beta}(z)(\tau(\bold{q}),\tX)) = [z^{m+1}](U^{\hat\beta}(z)(\tau(\bold{q}),\tX)).
\]
From the definition above, we have:
    \begin{equation*}
       \frac{\partial F_{0,1}(\tau(\bold{q}),1;\tX)}{\partial \tilde{q}_{i}}=\dsum_{\alpha}((\bold{1}_{1})^{\star i}-\frac{\bw_{i}}{\tilde{q}_{i}},S(\phi_{\alpha})(z))\tilde{\xi}^{\alpha}(z,\tX)[z^{-1}]\Big|_{Q=1}
    \end{equation*}
    for $i=-m,\dots,r-1$.
On the B-model side, we have:
\begin{lma}\label{lma:5.1}
Let $\{\bu^{\hat\alpha}(\bq)\}$ be the coordinates with respect to normalized canonical basis $\hat{\phi}_{\alpha}(\bq)$, we have:
    \begin{equation*}
    \frac{\partial\fh^\bullet_{X}(\Phi_{0})}{\partial u^{\hat\alpha}(\bq)}=\fh^\circ_{X}(\frac{-1}{\sqrt{2}}\xi_{\alpha,0}).
\end{equation*}
\end{lma}
\begin{proof}
Following the definition, we have
    \begin{equation*}
        \frac{\partial \fh^\bullet_{X}\big(\Phi_{0})}{\partial \tilde{q}_{i}}
        =\fh_{X}^\circ\big(-\frac{\partial x}{\partial\tilde{q}_{i}}\cdot (\frac{\partial x}{\partial y})^{-1}\big)
        =-\fh_{X}^\circ((Y^{i}-\frac{{\bf w}_{i}}{\tilde{q}_{i}})(\frac{\partial W_{T}}{\partial y})^{-1}\big)
    \end{equation*}
    for $i=-m,\dots,r-1$.
We have:
\begin{equation*}    
    \frac{\partial\fh^\bullet_{X}\big(\Phi_{0}\big)}{\partial u^{\hat\alpha}(\bq)}
    =-\fh^\circ_{X}\Big(\frac{\phi_{\alpha}(\bq)(\Delta^{\alpha}(\bq))^{\frac{1}{2}}}{\frac{\partial W_{T}}{\partial y}}\Big),
\end{equation*}
where $\phi_\alpha(\bq)\in\Jac(W_T(Y))$ is the unique representative such that $\phi_\alpha(\bq) =\sum_{i} c_iY^i$, $i\in\{-m,\dots,r-1\}$.

By Proposition \ref{prop:Frobenius}, we know
\begin{align*}    
&\phi_\alpha(\bq) = \prod_{\beta\neq \alpha}\frac{Y-P_\beta}{P_\alpha - P_\beta}, \quad \Delta^\alpha(\bq) = \frac{r\prod_{\beta\neq \alpha}(P_\alpha - P_\beta)}{P_\alpha^{m-1}},
\\
&\frac{\partial W_T}{\partial y}  = rY^{-m}\cdot \prod_{\beta}(Y-P_\beta).
\end{align*}
Since $\prod_{\beta}(Y-P_\beta)=0$ in $\Jac(W_T(Y))$, we take the equivalent representative for $\phi_\alpha(\bq)$:
$$
    \phi_\alpha(\bq) = \frac{P_\alpha^m\phi_\alpha(\bq)}{Y^m} = \frac{P_\alpha^m \prod_{\beta\neq \alpha }(Y-P_\beta)}{Y^m\prod_{\beta\neq \alpha }(P_\alpha-P_\beta)} \in \Jac(W_T(Y)),
$$
which is the unique representative such that $\phi_\alpha(\bq) =\sum_{i} c_iY^i$, $i\in\{-m,\dots,r-1\}$.
Therefore, we get
\[
    \fh^\circ_X\Big(\frac{\phi_{\alpha}(\bq)(\Delta^{\alpha}(\bq))^{\frac{1}{2}}}{\frac{\partial W_{T}}{\partial y}}\Big)=\fh^\circ_{X}(\frac{1}{\sqrt{2}}\xi_{\alpha,0}).\qedhere
\]
\end{proof}
Using Lemma \ref{lma:5.1}, we have:
for $\alpha\in\{0,1,\dots,m+r-1\}$, $k> 0$,
\begin{equation}\label{eqn:A-B-open}
    \begin{aligned}
        [z^k]\sum_{\beta}\tilde{\xi}^\beta(z,X)S(\hat{\phi}_\alpha(\bq),\phi_\beta)(z)\big|_{Q=1}&= \fh_X^\circ\Big(
        (-\frac{d}{dx})^k\Big(\frac{-1}{\sqrt{2}}\xi_{\alpha,0}\Big)\Big),
        \\
        \Big(\sum_{\beta}\tilde{\xi}^\beta(z,X)S(\hat{\phi}_\alpha(\bq),\phi_\beta)(z)\big|_{Q=1}\Big)_{+}&= \fh^\bullet_X\Big(\sum_{k\geq 0} 
        \frac{-1}{\sqrt{2}}W^\alpha_k z^k\Big)
        = \fh^\bullet_X\Big(\frac{-\hat{\theta}_\alpha(z)}{\sqrt{2}}\Big).
    \end{aligned}
\end{equation}
\begin{lma}\label{lma:theta-R-matrix}
    We have
    \[
        \theta_\alpha(z) = \sum_{\beta\in\{0,1,\dots,m+r-1\}}\check{R}_\beta^{\ \alpha}(-z)\hat{\theta}_\beta(z).
    \]
\end{lma}
\begin{proof}
    The lemma follows from
    \[
        d\xi^\alpha_{k} = \sum_{i=0}^{k}\sum_{\beta\in\{0,\dots,m+r-1\}}([z^{k-i}]\check{R}_\beta^{\ \alpha}(-z))W_i^\beta,
    \]
    which is shown in the proof of \cite[Lemma 6.7]{FLZ20a}.
\end{proof}
Then we have the identification of open leaves:
\begin{theorem}\label{thm:open-leaves}
    For each ordinary leaf $l_j$ with $\beta(l_j)=\beta$ and $k(l_j)=k\in\ZZ_{\geq 0}$, we have
    \[
        (\cL^O)^\beta_k(l_j)|_{Q=1}= (\check{\cL}^O)^\beta_k(l_j).
    \]
\end{theorem}
\begin{proof}
    By \cite{Lan25}, the A-model and B-model $R$-matrices are equal:
    $$R(z)=\check{R}(z).$$
    Then Theorem \ref{thm:open-leaves} follows from Equation \eqref{eqn:A-model-open-leaf}, Equation \eqref{eqn:A-B-open}, and Lemma \ref{lma:theta-R-matrix}.
\end{proof}
\subsection{Mirror symmetry for annulus invariants}
Define 
\begin{equation*}
    \begin{aligned}
        C(Y_1,Y_2) &:= (-\frac{\partial}{\partial x(Y_1)}-\frac{\partial}{\partial x(Y_2)})
        \Big(\frac{\omega_{0,2}}{dx(Y_1)dx(Y_2)}\Big)(Y_1,Y_2)dx(Y_1)dx(Y_2)
        \\
        & \ = \Big(-d_1\circ\frac{1}{dx(Y_1)}-d_2\circ\frac{1}{dx(Y_2)}\Big)(\tilde{\omega}_{0,2}(Y_1,Y_2)).
    \end{aligned}
\end{equation*}

The following proposition is proved in \cite[Lemma 6.9]{FLZ20b}.
\begin{prop} We have
\[
    C(Y_1,Y_2) = \frac{1}{2}\sum_{\alpha}d\xi^{\alpha}_{0}(Y_1)d\xi^{\alpha}_{0}(Y_2).
\]
\end{prop}

We define a rescaling operator on $W_{g,n}(\bq,X_1,\dots,X_n)$:
\begin{align*}
    {\rm D}_\bq : X_i^\mu \mapsto 
        \begin{cases}
            \displaystyle
            q_*^{\,\mu\frac{\bp_-}{\bp}}
            \prod_{\ell=-m+1}^{r-1}
            q_\ell^{\,\mu\frac{\bw_\ell}{\bp}}
            X_i^\mu,
            &\mu>0,\\[4mm]
            \displaystyle
            q_*^{\,\mu\left(\frac{\bp_-}{\bp}+1\right)}
            \prod_{\ell=-m+1}^{r-1}
            q_\ell^{\,\mu\frac{\bw_\ell}{\bp}}
            X_i^\mu,
            &\mu<0.
        \end{cases}
\end{align*}
and define
$$
    W_{0,2}(0,X_1,X_2) := \lim_{q_*,q_i\rightarrow 0} {\rm D}_{\bq}\Big(W_{0,2}(\bq,X_1,X_2)\Big).
$$

\begin{lma}
   \label{thm:annulus} Let $F_{0,2}(0;\tX_1,\tX_2)$ be the A-model annulus function at $\bt,Q=0$.
   Identify $\tX_i$ and $X_i$, we have
    \[
        F_{0,2}(0;\tX_1,\tX_2)=W_{0,2}(0,X_1,X_2).
    \]
\end{lma}
\begin{proof}
\begin{equation*}
    \begin{aligned}
    & \ \ \ \ \ \mathfrak{h}^\bullet_{X_1,X_2}(C)= 
    \mathfrak{h}^\bullet_{X_1,X_2}(\frac{1}{2}\sum_{\alpha}d\xi^{\alpha}_{0}(Y_1)d\xi^{\alpha}_{0}(Y_2))
    \\
    &= [z_1^0z_2^0]\sum_{\alpha,\beta,\gamma}
    \tilde{\xi}^\beta(z_1,\tX_1)\tilde{\xi}^\gamma(z_2,\tX_2)
    S(\hat{\phi}_\alpha(\bq),\phi_\beta)(z_1,1)S(\hat{\phi}_\alpha(\bq),\phi_\gamma)(z_2,1)
    \\
    &= [z_1^0z_2^0](z_1+z_2)\sum_{\beta,\gamma}V_{z_1,z_2}(\phi_\beta,\phi_\gamma)|_{Q=1}
    \tilde{\xi}^\beta(z_1,\tX_1)\tilde{\xi}^\gamma(z_2,\tX_2)
    \\
    &= \frac{1}{\bp}(\tX_1\frac{\partial}{\partial \tX_1}+\tX_2\frac{\partial}{\partial \tX_2})F_{0,2}(\tau(\bold{q}),1;\tX_1,\tX_2).
    \end{aligned}
\end{equation*}
By integrations by parts,
\begin{equation*}
    \begin{aligned}
        \fh^\bullet_{X_1,X_2}(C) &= \fh^\bullet_{X_1,X_2}\Big(\Big(-d_1\circ\frac{1}{dx(Y_1)}-d_2\circ\frac{1}{dx(Y_2)}\Big)\tilde{\omega}_{0,2}\Big) 
        \\
        &= \frac{1}{\bp}\Big(X_1\frac{\partial}{\partial X_1}+X_2\frac{\partial}{\partial X_2}\Big)W_{0,2}(\bold{q};X_1,X_2).
    \end{aligned}
\end{equation*}
It identifies the series $X_1^{\mu_1}X_2^{\mu_2}$ except $\mu_1=-\mu_2=\mu$. We transfer $X_i \rightarrow {\rm D}_{\bq}(X_i)$ and take the limit $q_*, q_l \rightarrow 0$.
Recall that $\tau=t H + \sum_{i=-m+1}^{r-1}t^i\one_i$, under the mirror map $\tau=\tau(\bq)$, $q_*,q_l$ tends to 0 is equivalent to say $t\rightarrow -\infty$ and $t^i\rightarrow 0$. 
The divisor equation tells that after the rescaling ${\rm D}_{\bq}$, $q_*$ records the curve class degrees as $Q$. We get 
$$
    \frac{1}{\bp}(X_1\frac{\partial}{\partial X_1}+X_2\frac{\partial}{\partial X_2})W_{0,2}(0,X_1,X_2) = \frac{1}{\bp}(\tX_1\frac{\partial}{\partial \tX_1}+\tX_2\frac{\partial}{\partial \tX_2})F_{0,2}(0;\tX_1,\tX_2).
$$
By Poincar\'e pairing $(,)_{\cF(m,r),T}$, $F_{0,2}(0,\tX_1,\tX_2)$ just contains the term $\tX_1^{\mu_1}\tX_2^{\mu_2}$ with $\mu_1\mu_2>0$. Meanwhile, $W_{0,2}$ also just contains the terms with same sign. The Theorem follows naturally.
\end{proof}

\subsection{Identification of graph sums}
The identification of open leaves shows the identification of graph sums. Combining this result with the identification $(0,1)$, $(0,2)$ case, we have the following theorem.
\begin{theorem} \label{thm:main}
Letting $X_{i}=e^{-\frac{W_{T}(Y_{i})}{\bp}}$, we have:
\[
    W_{g,n}(\bold{q},X_{1},\dots,X_{n})= F_{g,n} (\tau(\bold{q}),1;X_1,\dots,X_n)
\] 
for $g\geq 0, n>0$.
\end{theorem}
\begin{proof}
    For the unstable case $(g,n)=(0,1)$, this theorem is Theorem \ref{thm:disk}.

    For stable cases $2g-2+n>0$, it suffices to show the A- and B-model graph sum formula in Corollary \ref{thm:open-A-graph-sum-Qone} and Theorem \ref{thm:B-graph-sum} are equal (with sign corrections on ordinary leaves).
    By \cite[Proposition 3.2]{Lan25}, we have
    \[
        R_\beta^{\ \alpha}(z)\big|_{Q=1} = \check{R}_\beta^{\ \alpha}(z),
    \]
    and 
    \[
        \frac{h^\alpha_1}{\sqrt{2}} = \frac{1}{\sqrt{\Delta^\alpha(\bq)}}.
    \]
    Hence, the weights in the graph sum match except for the open leaves. Since 
    Theorem \ref{thm:open-leaves} identifies the A- and B- open leaves, the stable cases follow immediately.

    For the unstable case $(g,n)=(0,2)$, the special geometry \cite{EO07} tells that there are B-cycles $B_a$ on the mirror curve $C_q$ such that
    \[
        (\nabla_{\partial/\partial_{t^a}} B)(p_1,p_2) = \int_{p\in B_a} \omega_{0,3}(p_1,p_2,p).
    \]
    In A-model, let $\bt = \sum_a t^aT_a$, take the derivative $t^a$ on $F_{0,2}$ means that we are considering
    \[
        \frac{\partial}{\partial t^a}F_{0,2}(\bt,Q;X_1,X_2) = \sum_{k_1,k_2\geq 0\atop{\alpha_{i}\in \{0,\dots,m+r-1\}}}
        \llangle T_a,\phi_{\alpha_1}\psi^{k_1},\phi_{\alpha_2}\psi^{k_2}\rrangle_{0,3}^{\cF(m,r),T}
        \prod_{j=1}^{2}\tilde{\xi}_{k_j}^{\alpha_j}(X_j).
    \]
    Following \cite[Section 7.6]{FLZ20a} and the identification of A- and B-model graph sum formula, we know $$\partial_{t^a}(F_{0,2}(\tau(\bq),1;X_1,X_2)-W_{0,2}(\bq,X_1,X_2))=0.$$
    Lemma \ref{thm:annulus} identifies the initial values. The theorem follows immediately.
\end{proof}
\subsection{Full descendant open mirror symmetry}\label{sec:full descendant MS}
In this section we consider the Laplace transforms at appropriate cycles to Theorem \ref{thm:main} to produce an open mirror symmetry theorem concerning descendant potentials. We follow the notations and results in \cite[Section 4.2]{Lan25} and \cite{Fang20}.

\begin{defn}[equivariant Chern character] We define equivariant Chern character
\[\widetilde{ch}_z: K_{T}(\cF(m,r))\rightarrow H_{\CR,T}^\ast(\cF(m,r),\mathbb{Q})\left[\left[ \dfrac{p}{z} \right]\right]\] by the following two properties which uniquely characterize it:
\begin{enumerate}
  \item $\widetilde{ch}_z(\varepsilon_1\oplus\varepsilon_2)=\widetilde{ch}_z(\varepsilon_1)+\widetilde{ch}_z(\varepsilon_2)$.
  \item Let $\cX=\cF(m,r)$, $I\cX=\bigsqcup_{v\in\mathrm{Box}(\Sigma)}\cX_{v}$.If $\mathcal{L}$ is a $T$-equivariant line bundle on $\cF(m,r)$, then
  $$\widetilde{ch}_z(\mathcal{L})=\bigoplus_{v\in\mathrm{Box}(\Sigma)}\exp\big( 2\pi i(-\frac{(c_1)_T(\mathcal{L}_{v})}{z}+\age_{v}(\cL))\big)$$
  where $\cL_{v}=\cL|_{\cX_{v}}$, and $\age_{v}(\cL)$ is the age of $\cL_{v}$ along $\cX_{v}$.
  \item $\widetilde{\Gamma}_{z}(\cL)=\bigoplus_{v\in\mathrm{Box}(\Sigma)}(-z)^{1+\frac{(c_1)_T(\cL_{v})}{z}-\age_{v}(\cL)}\Gamma(1+\frac{(c_1)_T(\cL_{v})}{z}-\age_{v}(\cL))$
\end{enumerate}
\end{defn}

\begin{defn}[equivariant $K$-theoretic framing]
For $\forall \varepsilon\in K_T(\cF(m,r))$, we define the $K$-theoretic framing of $\varepsilon$ by
$\kappa(\varepsilon)= \widetilde\Gamma_{z}\big(T\cX\big) \widetilde{ch}_z(\varepsilon)$.
\end{defn}

\begin{defn}[equivariant SYZ $T$-dual] 
    Let $\mathcal{L}=\mathcal{O}_{\cF(m,r)}(\ell_1 p_1+\ell_2 p_2)$ be an equivariant ample line bundle on $\cF(m,r)$, where $\ell_1,\ell_2 \in \mathbb{Z}$, such that $\ell_1+\ell_2 >0$. We define equivariant Strominger-Yau-Zaslow (SYZ) T-dual \cite{SYZ96} $\mathrm{SYZ}(\mathcal{L})$ of $\mathcal{L}$ be the oriented path in $\bC$ (see Figure \ref{fig:SYZmirror}). We extend this definition additively to $K_{T}(\cF(m,r))$.
\end{defn}

\begin{figure}[h]
\begin{center}
\setlength{\unitlength}{2mm}
\begin{picture}(60,20)

\put(10,6){\vector(1,0){10}}
\put(20,6){\line(1,0){10}}
\put(30,6){\vector(0,1){5}}
\put(30,11){\line(0,1){5}}
\put(30,16){\vector(1,0){10}}
\put(40,16){\line(1,0){10}}

\put(5,2){$-\infty+2\pi i \cdot \frac{-\ell_1}{m}$}
\put(27,2){$2\pi i\cdot\frac{-\ell_1}{m}$}
\put(27,18){$2\pi i\cdot\frac{\ell_2}{r}$}
\put(40,18){$2\pi i \cdot\frac{\ell_2}{r}+(+\infty)$}

\end{picture}
   \caption{SYZ $T$-dual of $\cL$ in $\bC$}
  \label{fig:SYZmirror}
\end{center}
\end{figure}
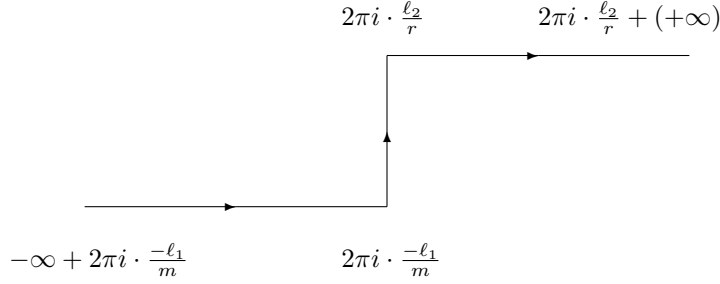

In \cite{Lan25}, we have the following results:
\begin{prop}\label{T-dual transformed}
\

\begin{enumerate}
  \item By string equation: \[
  \dint_{\mathrm{SYZ}(\mathcal{L})} e^{\frac{W_T}z} dy =\left<\!\left< \dfrac{\kappa(\mathcal{L})}{z(z-\psi)}\right>\!\right>^{\cF(m,r),T}_{0,1}
  =\left<\!\left< 1, \dfrac{\kappa(\mathcal{L})}{z-\psi}\right>\!\right>^{\cF(m,r),T}_{0,2};
  \]
  \item Integrating by parts, \[ -\left<\!\left< \dfrac{\kappa(\mathcal{L})}{z-\psi}\right>\!\right>^{\cF(m,r),T}_{0,1} =-z \dint_{\mathrm{SYZ}(\mathcal{L})} e^{\frac{W_T}z} dy = \dint_{\mathrm{SYZ}(\mathcal{L})} e^{\frac{W_T}z} ydx.
  \]
\end{enumerate}
\end{prop}

Define
\[
S_{\underline{\beta}}^{\ \underline{\hat{\alpha}}}(z)=\left<\!\left< \phi_\beta(\bold{q}),\frac{\hat{\phi}_\alpha(\bold{q})}{z-\psi}\right>\!\right>_{0,2}^{\cF(m,r),T},\\
S_{\widehat{\underline{\beta}}}^{\ \kappa(\mathcal{L})} (z) =\left<\!\left< \hat{\phi}_\beta(\bold{q}),\frac{\kappa(\mathcal{L)}}{z-\psi} \right>\!\right>_{0,2}^{\cF(m,r),T}.\]
More generally, we have
\begin{prop}
\begin{align*}
    S^{\ \alpha}_{\widehat{\underline{\beta}}}(z) &=  -z \dint_{y\in \gamma_\alpha} e^{\frac{W_T}z} \dfrac{d\xi^{\beta}_{0}}{\sqrt{2}}.
    \\
    S_{\widehat{\underline{\beta}}}^{\ \kappa(\mathcal{L})}(z)&=-z\dint_{y\in{\mathrm{SYZ}(\mathcal{L})}} e^{\frac{W_T}z}\dfrac{d\xi^{\beta}_{0}}{\sqrt{2}}.
\end{align*}
\end{prop}
Integrating the second equation by parts, we have
\[S_{\widehat{\underline{\beta}}}^{\ \kappa(\mathcal{L})} (z) =-z^{k+1}\dint_{y\in \mathrm{SYZ}(\mathcal{L})}  e^{\frac{W_T}z} \frac{W_k^\beta}{\sqrt{2}}.\]
Also notice that
\[\begin{split}
&\dsum_{\gamma=0}^{m+r-1} S_{\ \alpha}^{\widehat{\underline{\gamma}}} (z) S_{\ \beta}^{\widehat{\underline{\gamma}}}(-z) = (\phi_\alpha(0),\phi_\beta(0))=\frac{\delta_{\alpha\beta}}{\Delta^{\alpha}(0)},\\
&\dsum_{\alpha=0}^{m+r-1} S_\beta^{\ \widehat{\underline{\alpha}}} (-z) S_{\widehat{\underline{\alpha}}}^{\ \kappa(\mathcal{L})}(z) = (\phi_\beta(0),\mathcal{\kappa(L)}).
\end{split}
\]
Let $g,n_1,n_2\in\bZ_{\geq 0}, 2g-2+n_1+n_2> 0$, we define
\begin{equation*}
    W_{g,(n_{1},n_{2})}(y_{1},\dots,y_{n_{1}},X_{n_{1}+1},\dots,X_{n_{1}+n_{2}}):=\fh^\bullet_{X_{n_{1}+1},\dots,X_{n_{1}+n_{2}}}({\omega}_{g,n_{1}+n_{2}}).
\end{equation*}
And we have the following lemma
\begin{lma}\label{lma:5.3}
   Letting $X=e^{-W_{T}(Y)/\bp}$, and $Y=e^y$, we have
   \begin{equation*}       
   \txi^{\beta}_{k}(X)=\fh^\bullet_{X}\big(\Delta^{\beta}(0)\dsum_{\alpha=0}^{m+r-1}\dsum_{l\in\mathbb{Z}_{\geq 0}} [z^{k-l}] S_{\ \beta}^{\widehat{\underline{\alpha}}} (-z) \dfrac{-W_l^\alpha(y)}{\sqrt{2}} \big).
   \end{equation*}  
\end{lma}
\begin{proof}
    This directly comes from the fact that:
    \begin{equation*}        
        \Big(\sum_{\beta}\tilde{\xi}^\beta(z,X)S(\hat{\phi}_\alpha(\bq),\phi_\beta)(z)\Big|_{Q=1}\Big)_{+}= \fh^\bullet_X\Big(\sum_{k\geq 0} 
        \frac{-1}{\sqrt{2}}W^\alpha_k z^k\Big)
    \end{equation*}
    and 
    \begin{equation*}
        \dsum_{\alpha=0}^{m+r-1}S^{\hat{\ualpha}}_{\ \beta}(z) S^{\hat{\ualpha}}_{\ \gamma}(-z)=\frac{\delta_{\beta,\gamma}}{\Delta^{\beta}(0)}.
    \end{equation*}
\end{proof}
Using Lemma \ref{lma:5.3}, we can prove the following theorem.
\begin{theorem}[full descendant open mirror symmetry]\label{thm:full descendant MS}
    For $g,n_1,n_2\geq 0, n_1+n_2>0$, $2g-2+n_{1}+n_{2}>0$, and $\cL_{i}\in K_{T}(\cF(m,r))$ for $i = 1,\dots, n_1$, let $Q=1$, we have 
    \begin{eqnarray*}        &\dint\dots\dint_{y_{i}\in\mathrm{SYZ}(\cL_{i})} \exp(\dsum_{i=1}^{n_{1}}\frac{W_{T}(y_{i})}{z_{i}}) W_{g,(n_{1},n_{2})}(y_{1},\dots,y_{n_{1}},X_{n_{1}+1},\dots,X_{n_{1}+n_{2}}) \\
    & =\dsum_{\vec{\mu}\in(\bZ_{\neq0})^{n_{2}}}\left<\!\left< \dfrac{\kappa(\mathcal{L}_1)}{z_1-\psi_1}, \cdots,\dfrac{\kappa(\mathcal{L}_{n_{1}})}{z_{n_{1}}-\psi_{n_{1}}}  \right>\!\right>_{g,\vec{\mu}}^{\OGW}\dprod_{i=1}^{n_{2}}X_{i+n_{1}}^{\mu_{i}}.
    \end{eqnarray*}
\end{theorem}
\begin{proof}   
\[\begin{split}
& \dint\dots\dint_{y_{i}\in\mathrm{SYZ}(\cL_{i})} \exp(\dsum_{i=1}^{n_{1}}\frac{W_{T}(y_{i})}{z_{i}}) W_{g,(n_{1},n_{2})} \\
= & \dint\dots\dint_{y_{i}\in\mathrm{SYZ}(\cL_{i})} \exp(\dsum_{i=1}^{n_{1}}\frac{W_{T}(y_{i})}{z_{i}}) \dsum_{\beta_i,k_{i}}\Bigg[\left<\!\left<\dprod_{i=1}^{n_{1}+n_{2}} \tau_{k_{i}}(\phi_{\beta_i}(0)) \right>\!\right>_{g,n_{1}+n_{2}}^{\CGW} \\
& \cdot \dprod_{i=1}^{n_{1}} \big( \Delta^{\beta_i}\dsum_{\alpha=0}^{m+r-1}\dsum_{k\in\mathbb{Z}_{\geq 0}} [z_i^{k_{i}-k}] S_{\ \beta_i}^{\widehat{\underline{\alpha}}} (-z_i) \dfrac{-W_k^\alpha(y_i)}{\sqrt{2}} \big)\cdot\dprod_{i=n_{1}+1}^{n_{1}+n_{2}}\txi^{\beta_{i}}_{k_{i}}(\tX_{i}) \Bigg]\\
= & \dsum_{\beta_i,k_{i}}\left<\!\left<\dprod_{i=1}^{n_{1}+n_{2}} \tau_{k_{i}}(\phi_{\beta_i}(0)) \right>\!\right>_{g,n_{1}+n_{2}}^{\CGW}  \dprod_{i=1}^{n_{1}}  \Delta^{\beta_i}(\phi_{\beta_i}(0),\kappa(\mathcal{L}_i))z_i^{-k_{i}-1}\dprod_{i=n_{1}+1}^{n_{1}+n_{2}} \txi^{\beta_{i}}_{k_{i}}(\tX_{i})\\
= & \dsum_{k_{i}}\left<\!\left< \dfrac{\kappa(\mathcal{L}_1)}{z_1-\psi_1}, \cdots,\dfrac{\kappa(\mathcal{L}_{n_{1}})}{z_{n_{1}}-\psi_{n_{1}}}  ,\dprod_{i=1}^{n_{2}}\tau_{k_{i}}\Big(\dsum_{\beta}\txi^{\beta}_{k_{i}}(\tX_{n_{1}+i})\phi_{\beta}\Big)\right>\!\right>_{g,n_{1}+n_{2}}^{\CGW}\\
= & \dsum_{\vec{\mu}\in(\bZ_{\neq0})^{n_{2}}}\left<\!\left< \dfrac{\kappa(\mathcal{L}_1)}{z_1-\psi_1}, \cdots,\dfrac{\kappa(\mathcal{L}_{n_{1}})}{z_{n_{1}}-\psi_{n_{1}}}  \right>\!\right>_{g,\vec{\mu}}^{\OGW}\dprod_{i=1}^{n_{2}}X_{i+n_{1}}^{\mu_{i}}.
\end{split}\]
The last equation holds because of Theorem \ref{thm:open-descendant}.
\end{proof}

\end{document}